\documentclass[letterpaper,notitlepage,11pt,twoside]{article} 

\usepackage{amsmath} 
\usepackage{amssymb} 
\usepackage{amsthm} 
\usepackage{bbm} 

\numberwithin{equation}{section}
\usepackage{multicol}
\usepackage[backend=bibtex,maxnames=15]{biblatex} 
\bibliography{biblio_LWC}

\AtEveryBibitem{%
    \clearfield{url}
    \clearfield{urldate}%
    \clearfield{doi}%
		\clearfield{eprint}%
		\clearfield{ISSN}
}
\renewbibmacro{in:}{}
\usepackage{graphicx} 

\usepackage{subcaption} 

\usepackage{framed} 

\usepackage{enumerate} 
\usepackage{enumitem}
\setlist[enumerate]{itemsep = -0.2em}

\usepackage{authblk} 

\usepackage[explicit]{titlesec}  
\titleformat{\subsection}[runin]
  {\normalfont\normalsize\bfseries}{\thesubsection}{0.3em}{#1.}
	
\titleformat{\section}
  {\normalfont\large}{\thesection}{0.2em.}{\centering \textsc{#1}}

\usepackage{caption}  
\usepackage[titletoc, title]{appendix} 

\usepackage{hyperref} 
\usepackage{color} 

\definecolor{MyDarkBlue}{rgb}{0,0.08,0.50}  
\definecolor{BrickRed}{rgb}{0.65,0.08,0}

\hypersetup{  
colorlinks=true,       
    linkcolor=MyDarkBlue,          
    citecolor=BrickRed,        
    filecolor=red,      
    urlcolor=cyan           
}

\usepackage[headheight=0.6cm,headsep=0.6cm,margin=1in,bottom=1.1cm]{geometry} 

\usepackage{fancyhdr}
\newtheorem{Lemma}{Lemma}[section]

\newtheorem{Proposition}[Lemma]{Proposition}

\newtheorem{Theorem}[Lemma]{Theorem}

\newtheorem{Remark}[Lemma]{Remark}

\newtheorem{Definition}[Lemma]{Definition}
\newtheorem{Condition}[Lemma]{Condition}

\newtheorem{remark}[Lemma]{Remark}

\newcommand{\sub}[1]{\boldsymbol{#1}}
\newcommand{\R}{\mathbb{R}}

\newcommand{\N}{\mathbb{N}}

\newcommand{\I}{\mathbbm{1}}
\newcommand{\pr}{\mathbb{P}}
\newcommand{\E}{\mathbb{E}}

\newcommand{\G}{\mathcal{G}}
\newcommand{\D}{\mathcal{D}}
\newcommand{\U}{\mathcal{U}}

\renewcommand{\P}{\mathcal{P}}

\newcommand{\emp}{\varnothing}

\newcommand{\eqn}[1]{\begin{equation} #1 \end{equation}}

\newcommand{\T}{\mathcal{T}}
\newcommand{\Id}{\mathrm{Id}}

\newcommand{\PA}{\mathrm{PA}}
\newcommand{\sss}{\scriptscriptstyle}

\newcommand{\din}{d^{\sss(\mathrm{in})}}
\newcommand{\dout}{d^{\sss(\mathrm{out})}}
\newcommand{\mout}{m^{\sss(\mathrm{out})}}
\newcommand{\Din}{D^{\scriptscriptstyle(\mathrm{in})}}
\newcommand{\Dout}{D^{\scriptscriptstyle(\mathrm{out})}}

\newcommand{\subDout}{\sub{D}^{\sss(\mathrm{out})}}
\newcommand{\subDin}{\sub{D}^{\sss(\mathrm{in})}}
\newcommand{\DDin}{\mathcal{D}^{\scriptscriptstyle(\mathrm{in})}}
\newcommand{\DDout}{\mathcal{D}^{\scriptscriptstyle(\mathrm{out})}}

\newcommand{\Win}{W^{\sss(\mathrm{in})}}
\newcommand{\Wout}{W^{\sss(\mathrm{out})}}

\begin{document}

\title{\bfseries\uppercase{\large Local weak convergence for PageRank} }

\author[a,1]{Alessandro Garavaglia}
\author[a,2]{Remco van der Hofstad}
\author[a,b,3]{Nelly Litvak}
\affil[a]{\footnotesize Department of Mathematics and
    Computer Science, Eindhoven University of Technology, 5600 MB Eindhoven, The Netherlands}
\affil[b]{\footnotesize Department of Applied Mathematics, Faculty of Electrical Engineering, Mathematics and Computer Science,
				 University of Twente, 7500 AE Enschede, The Netherlands}

\affil[$ $]{ {\itshape Email address}: $^1$a.garavaglia@tue.nl, $^2$rhofstad@win.tue.nl, $^3$n.litvak@utwente.nl, \\

\vspace{0.4cm}
\textsc{keywords}: PageRank, local weak convergence, directed random graphs}

\date{}
\maketitle

\thispagestyle{plain}

\vspace{-1cm}
\begin{abstract}
PageRank is a well-known algorithm for measuring centrality in networks. It was originally proposed by Google for ranking pages in the World-Wide Web. One of the intriguing empirical properties of PageRank is the so-called `power-law hypothesis': in a scale-free network the PageRank scores follow a power law with the same exponent as the (in-)degrees. Up to date, this hypothesis has been confirmed empirically and in several specific random graphs models. In contrast, this paper does not focus on one random graph model but investigates the existence of an asymptotic PageRank distribution, when the graph size goes to infinity, using local weak convergence. This may help to identify general network structures in which the power-law hypothesis holds. We start from the definition of local weak convergence for sequences of (random) undirected graphs, and extend this notion to directed graphs. To this end, we define an exploration process in the directed setting that keeps track of in- and out-degrees of vertices. Then we use this to prove the existence of an asymptotic PageRank distribution. As a result, the limiting distribution of PageRank can be computed directly as a function of the limiting object. 
We apply our results to the directed configuration model and continuous-time branching processes trees, as well as preferential attachment models.

\end{abstract}

\section{Introduction and main results}
\label{sec-intro}

\subsection{Definition of PageRank}  PageRank, first introduced in \cite{PageRank}, is an algorithm that generates a centrality measure on finite graphs. Originally introduced to rank World-Wide Web pages, PageRank has a wide range of applications including citation analysis \cite{Chen,MaGuan,Walt3}, community detection \cite{Andersen} or social networks analysis \cite{Bahmani,Wang2013}.  

Consider a finite directed (multi-)graph $G$ of size $n$. We write $[n] = \{1,\ldots,n\}$. Let $e_{j,i}$ be the number of directed edges from $j$ to $i$. Denote the in-degree of vertex $i\in[n]$ by $\din_i$  and the out-degree  by $\dout_i$ . Fix a parameter $c\in (0,1)$, which is called the {\em damping factor}, or teleportation parameter. PageRank is the unique vector $\sub{\pi}(n) = (\pi_1(n),\ldots,\pi_n(n))$ that satisfies, for every $i\in[n]$,
\eqn{
\label{eq:for:pagerank:formulation}
	\pi_i(n) = c\,\sum_{j\in[n]}\frac{e_{j,i}}{\dout_j}\,\pi_j(n)+\frac{1-c}{n}.
}
PageRank has the natural interpretation as the invariant measure of a random walk with restarts on $G$. With probability $c$ the random walk takes a simple random walk step on $G$, while with probability $(1-c)$ it moves to a uniformly chosen vertex. Here by simple random walk we mean the random walk that chooses, at every step, an outgoing edge from the current position uniformly at random. When $\dout_j>0$ for all $j\in[n]$, then the invariant measure of this random walk is given exactly by \eqref{eq:for:pagerank:formulation}. The interpretation is easily extended to the case when some vertices $j$ have $\dout_j=0$ by introducing a random jump from such vertices; in this case the stationary distribution will be the solution of \eqref{eq:for:pagerank:formulation} renormalized to sum up to one~\cite{lee2003two-stage}.

In this paper we consider the {\em graph-normalized version} of PageRank, which is the vector defined as $\sub{R}(n) = n \sub{\pi}(n)$. We call both the algorithm and the vector $\sub{R}(n)$ PageRank,  the meaning will always be clear from the context. The graph-normalized version of (\ref{eq:for:pagerank:formulation}) is the unique solution $\sub{R}(n)$ to
\eqn{
\label{for:pagerank:formulation:normn}
	R_i(n) = c\,\sum_{j\in[n]}\frac{e_{j,i}}{\dout_j}\,R_j(n)+(1-c).
}
PageRank has numerous generalizations. For example, after a random jump, the random walk might not restart from a uniformly chosen vertex, but rather choose vertex $i$ with probability $b_i$, where $\sum_{i=1}^nb_i=1$.  Equation 
\eqref{eq:for:pagerank:formulation} then becomes
\eqn{
\label{eq:for:pagerank:formulation:pers}
	R_i(n) = c\,\sum_{j\in[n]}\frac{e_{j,i}}{\dout_j}\,R_j(n)+(1-c)b_i.
}
This generalized version of PageRank is sometimes called {\it topic-sensitive}~\cite{Haveliwala2002PPR} or {\it personalized}. We note that the term {\it personalized PageRank} often refers to the case when the vector ${\bf b}=(b_1,\ldots, b_n)$ has one of its coordinates equal to 1, and the rest equal to zero, so that the random walk always restarts from the same vertex. One can generalize further, e.g., allow the probability $c$  to be random as well. The literature~\cite{Litvak2,Jelenkovic2010WBP,Lee2017PR-IRG,Litvak4} usually studies the following graph-normalized equation:
\eqn{
\label{for:recursive:pagerank:pers:norm}
	R_i(n) = \sum_{j: e_{j,i}\ge 1}A_j R_j(n) + B_i,\quad i\in[n],
}
where $(A_i)_{i\in[n]}$ and $(B_i)_{i\in[n]}$ are values assigned to the vertices in the graph. In this paper, for simplicity of the argument, we will focus on the basic model \eqref{for:pagerank:formulation:normn} and then, in Section~\ref{sec:generalized},  extend the results to the more general model \eqref{for:recursive:pagerank:pers:norm} with $A_j=C_j/\dout_j$, where $C_j$'s are random variables bounded by $c<1$, and $(B_i)_{i\in[n]}$ are i.i.d.\ across vertices.

\begin{figure}[hb]
\centering
	\includegraphics[width = 0.35\textwidth]{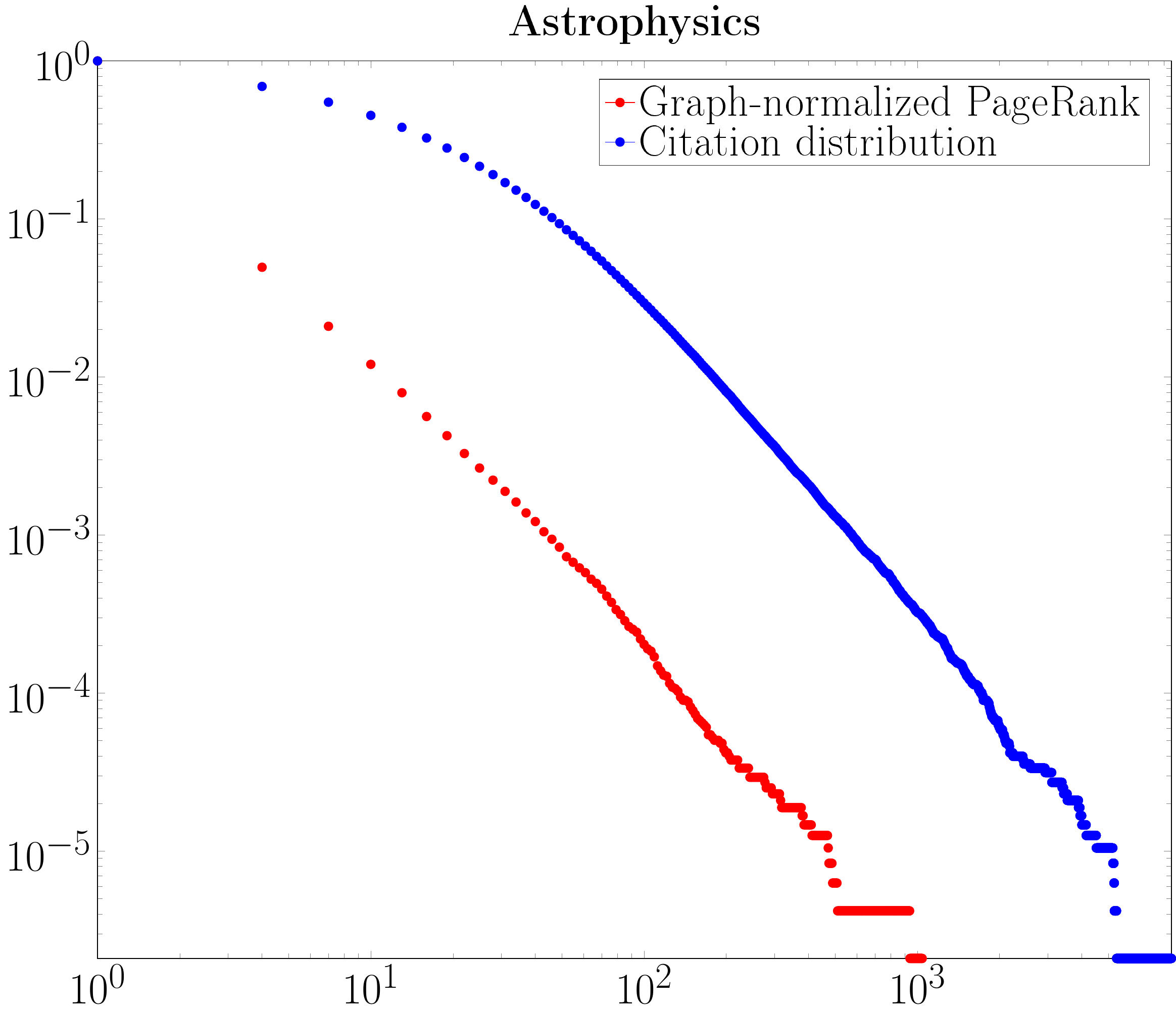} \hspace{1cm}
		\includegraphics[width = 0.35\textwidth]{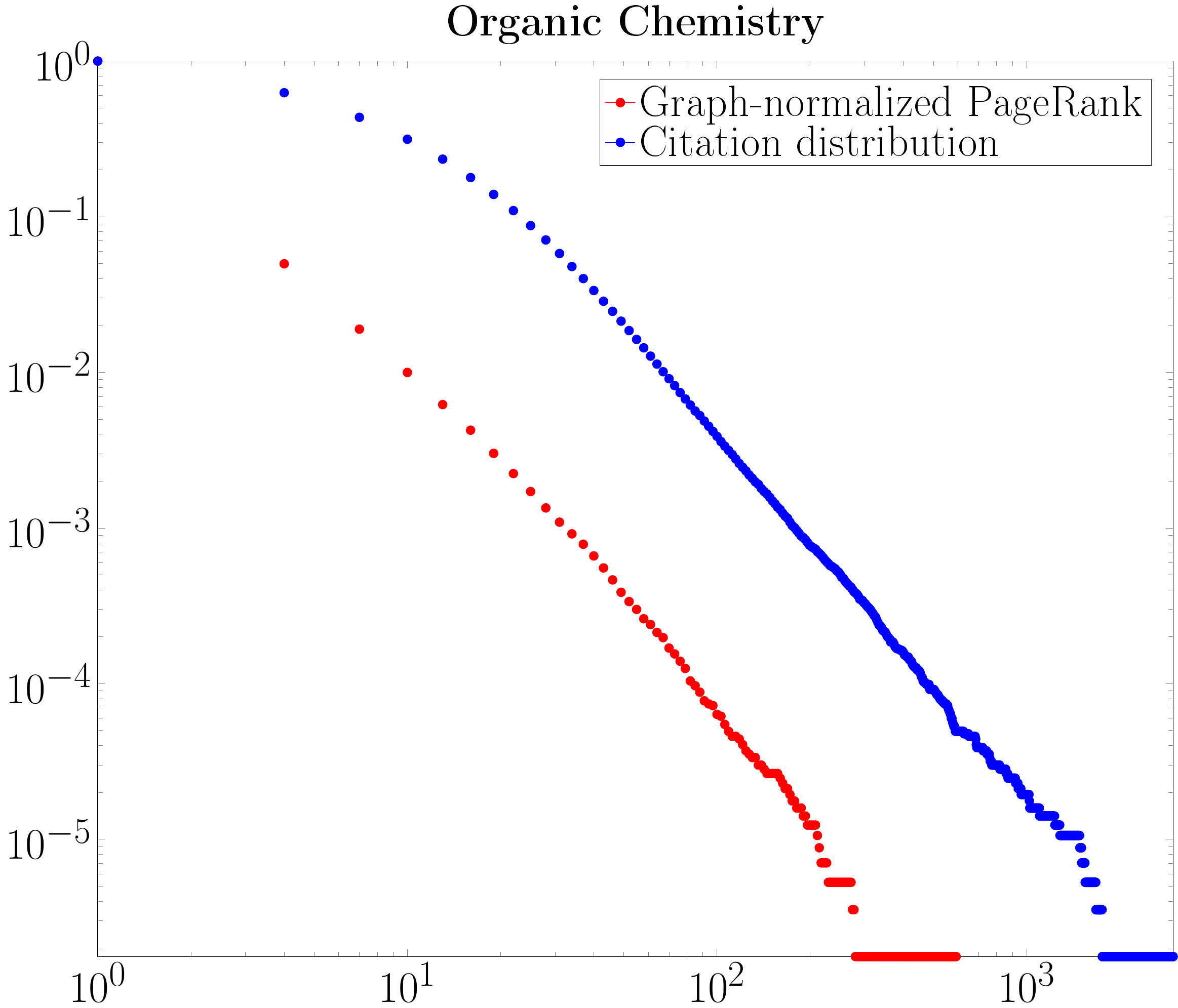} 
		\caption{Citation networks from {\em Web of Science}. Citation networks can be seen as directed graphs where references are directed edges. We considered papers from Astrophysics (left) and Organic Chemistry (right). The two loglog scale plots show two different distributions each. The blue data represent the tail distribution of the in-degree  (so number of citations) of a uniformly chosen vertex. The red data represents the tail distribution of the graph-normalized PageRank of a uniform vertex. Notice that, in both cases, the two distributions show a remarkably similar power-law exponent.}
		\label{fig:pagerank:degree}
\end{figure}

\subsection{ Power-law hypothesis for PageRank}
\label{sec:pagerank:powerlaw}
It has been observed \cite{litvak5,Pandurangan} that in real-world networks with power-law (in-)degree distributions, PageRank follows a power-law with the same exponent. Figure \ref{fig:pagerank:degree} illustrates this phenomenon in citation networks. Empirical studies suggest that the power-law hypothesis holds rather generally. However, proving this appears to be challenging. Some progress has been made in \cite{Avrac} for the average PageRank in a Preferential Attachment model. In a series of papers \cite{litvak5,Litvak4}, the result was proved provided that PageRank satisfies a branching-type of recursion. Then, in fact, a network is modeled as a branching process, with independent labels representing out-degrees. The full proof has been obtained in \cite{Litvak2} for the directed configuration model and recently in \cite{Lee2017PR-IRG} for directed generalized random graphs.

The motivation of this paper is in finding general conditions for the existence of an asymptotic PageRank distribution. We prove the convergence (in some sense) of PageRank for a large class of models. Our results also shed light on the power-law hypothesis. Indeed,  when the limit is a branching tree, this directly implies the power-law hypothesis, based on the above mentioned results in the literature. When the limit is different, e.g., the tree generated by a continuous-time branching process, proving the power-law hypothesis remains an open problem. Our results imply however that it is sufficient to study PageRank  on the limiting object, which hopefully is simpler since the graph-size asymptotics no longer interfere.

\subsection{Overview of the paper} In Section~\ref{sec:main_result} we explain our general methodology and the main ideas behind the proofs and the results. In Section~\ref{sec:LWCintro} we introduce the notion of local weak convergence, which is crucial in our approach, and explain how we extend it to directed graphs. Section~\ref{sec:PRproof_intro} describes the major steps in proving weak convergence for PageRank. In Section~\ref{sec:examples_intro} we present three examples that illustrate our results: directed configuration models (DCMs) (and the extension to directed inhomogeneous random graphs), continuous-time branching processes (CTBPs), and (directed) preferential attachment models (DPAs). In Section~\ref{sec:open_problems}, we list some open problems.

Sections~\ref{sec:LWC} -- \ref{sec:examples} contain formal proofs. In Section \ref{sec:LWC} we explain local weak convergence for undirected graph sequences (Section \ref{sec:LWC:und}) and introduce our construction for directed graph sequences (Section \ref{sec:LWC:dir}), which is taylored to our PageRank application. In section \ref{sec:generalized} we study generalized PageRank. In Section \ref{sec:pagerank} we formally prove the main result. The three examples, DCM (and Inhomogeneous random graphs), CTBPs and DPA, are analyzed, respectively, in Sections \ref{sec:ex:DCM} (and \ref{sec-ex-IRG}), \ref{sec:ex:CTBP} and \ref{sec:examples:PAM}.

\section{Main result and methodology}
\label{sec:main_result}

Note that for any deterministic graph, PageRank  is a deterministic vector. We are interested in the PageRank associated to {\em random} graphs. In particular, we want to investigate the asymptotic behavior of the PageRank value of a uniformly chosen vertex $V_n$, as the size of the graph grows. In this case we have two sources of randomness: the choice of the vertex and the randomness of the graph itself. Our main result shows that, for a nice enough sequence of directed graphs $(G_n)_{n\in\N}$, $R_{V_n}(n)$ converges in distribution to a limiting random variable:

\begin{Theorem}[Existence of asymptotic PageRank distribution]
\label{th:pagerank}
Consider a sequence of directed random graphs $(G_n)_{n\in\N}$. Then, the following hold:
\begin{enumerate}
	\item  If $G_n$  {\em converges  in distribution in the local weak sense}, then there exists a limiting distribution $R_\emp$, with $\E[R_\emp]\leq 1$, such that
$$
	R_{V_n}(n)\stackrel{d}{\displaystyle \longrightarrow}R_\emp;
$$
	\item If $G_n$  {\em converges  in probability in the local weak sense}, then there exists a limiting distribution $R_\emp$, with $\E[R_\emp]\leq 1$, such that,  for every $r>0$, 
	$$
		\frac{1}{n}\sum_{i\in[n]}\I\{R_i(n)>r\}\stackrel{\pr}{\longrightarrow}\pr\left(R_\emp>r\right).
	$$
\end{enumerate}
\end{Theorem}

Theorem \ref{th:pagerank} establishes that, whenever a sequence of directed random graphs converges in the {\em local weak sense}, then the distribution $R_{V_n}(n)$ admits a limit in distribution, $R_\emp$. This limit has the interpretation of PageRank on the (possibly infinite) limiting graph.  
Theorem \ref{th:pagerank} can be extended to personalized PageRank defined in \eqref{for:recursive:pagerank:pers:norm} under additional conditions on the random variables $(A_i)_{i\in\N}$ and $(B_i)_{i\in\N}$. The precise formulation, that requires more notation, is given in Theorem \ref{th:pagerank:pers}.

\begin{Remark}[Stochastic lower bound for PageRank]
\label{rem:stoch:lowerbound}
{\em Theorem \ref{th:pagerank} gives a rough lower bond on the tail of the asymptotic PageRank distribution for a graph sequence. In simple words, we can write 
\eqn{
\label{for:stoch:lowerbound}
	R_\emp \geq (1-c)\bigg(1+c\sum_{i=1}^{\Din_\emp}\frac{1}{\mout_i}\bigg),
}
where $\emp$ is a vertex called root in the local weak limit of the graph sequence $(G_n)_{n\in\N}$, $\Din_\emp$ is the graph limiting in-degree distribution, and $\mout_i$ represent the out-degree in the LW limit. All the notation in \eqref{for:stoch:lowerbound} is introduced in Sections \ref{sec:LWC:dir} and \ref{sec:pagerank}. In particular, \eqref{for:stoch:lowerbound} implies that $R_\emp>1-c$ a.s.. Since $\mout$ represents the limiting out-degree distribution, it follows that, if $(G_n)_{n\in\N}$ has out-degrees uniformly bounded by a constant $A<\infty$,   
$$
	R_\emp \geq (1-c)\bigg(1+\frac{c}{A}\Din_\emp\bigg).
$$
As a consequence, if the limiting in-degree distribution obeys a power law, then the tail of the distribution $R_\emp$ is bounded from below by a multiple of the tail of the in-degree. This establishes a power-law lower bound for $R_\emp$. This is a partial solution of the power-law hypothesis mentioned in Section \ref{sec:pagerank:powerlaw}.}
\end{Remark}

We next explain the ingredient of our main result, which is local weak convergence.

\subsection{Local weak convergence for directed graphs}
\label{sec:LWCintro}
Local weak (LW) convergence is a concept that was first introduced in  \cite{Aldous2007,Aldous2004,benjamini} for undirected graphs. In this framework, a sequence of undirected random graphs, under relatively weak conditions, converges to a (possibly random) {\em rooted graph}, i.e., a graph where one of the vertices is labeled as root. In simple words, the limiting graph resembles the {\em neighborhood of a typical vertex} in the graph sequence. This methodology has been shown to be useful to investigate {\em local} properties of a graph sequence -- the properties that depend on the {\em local neighborhood} of vertices.

In the literature, limits of different types of random graphs have been investigated (Aldous and Steele give a survey in \cite{Aldous2007}).   Grimmett \cite{Grimm80} obtained the LW limit for the uniform random tree. Generalized random graphs \cite{Britton,ChungLu2,ChungLu,HoogVdH} also converge in the LW sense under some regularity conditions on the weight distribution. Convergence of undirected configuration model is proved in \cite[Chapter 2]{vdH2}. In many random graph contexts, the LW limit is a branching process, and LW convergence provides a method to compare neighborhoods in random graphs to branching processes.

A recent work by Berger et al. \cite{BergerBorgs} investigates the LW limit for preferential attachment models, in the case of fixed number of edges and no self-loops allowed. In particular, their proof  covers the case of a power-law distribution with exponent $\tau\geq3$. Dereich and Morters \cite{Der2011,Der2009,Der2013} establish the LW limit in the case of preferential attachment models with conditionally independent edges. 

In the local weak convergence setting, a sequence of graphs $(G_n)_{n\in\N}$ converges to a (possibly) random rooted graph $(G,\emp)$ that is a {\em rooted graph}. Here $\emp\in V(G)$ denotes the root.

Heuristically, $G_n\rightarrow (G,\emp)$ in the LW sense when the law of the neighborhood of a typical vertex in $G_n$ converges to the law of the neighborhood of the root in $G$. We give now an intuitive formulation of this concept (for a precise definition, see Section \ref{sec:LWC:und}).
For a vertex $i$ in a graph $G_n$, denote the neighborhood of $i$ up to distance $k$ by ${U}_{\leq k}(i)$.  Then, for a random rooted graph $(G, \emp)$, we say that $G_n\rightarrow G$ if, for any finite rooted graph $(H,y)$, and any $k\in\N$,
\eqn{
\label{for:heur}
	\frac{1}{n}\sum_{i\in[n]}\I\{U_{\leq k}(i) \cong (H,y)\} \longrightarrow \pr\left(U_{\leq k}(\emp) \cong (H,y)\right),
}
where $\I\{\cdot\}$ is an indicator of event $\{\cdot\}$, and $U_{\leq k}(\emp)$ is the $k$-neighborhood of $\emp$ in $G$.  The event $\{U_{\leq k}(i) \cong (H,y)\}$  means that the $k$ neighborhood of $i$ is structured as $(H,y)$, ignoring the precise labeling of the vertices.  Notice that the left-hand term in \eqref{for:heur} is just the probability that the $k$-neighborhood of a uniformly chosen vertex in $G_n$ is structured as $(H,y)$. 

\eqref{for:heur} is formulated for a deterministic graph sequence $(G_n)_{n\in\N}$. When $(G_n)_{n\in\N}$ is a sequence of {\em random} graphs, the  left-hand term in \eqref{for:heur} is a random variable.
In this case there are different modes of convergence, as stated in Definition \ref{def:LWC:und:random}.  For example, we say that $G_n\rightarrow (G,\emp)$ in probability if, for any finite rooted graph $(H,y)$, and any $k\in\N$,
\eqn{
\label{for:heur:2}
	\frac{1}{n}\sum_{i\in[n]} \I\{U_{\leq k}(i) = (H,y)\} \stackrel{\pr}{\longrightarrow} \pr\left(U_{\leq k}(\emp) = (H,y)\right).
}

In Section \ref{sec:LWC:dir} we will extend the definition of LW convergence to directed graphs. Here we introduce the main ideas behind the construction.

The major problem in the construction is that the {\em exploration of neighborhoods} is not uniquely defined in directed graphs. Indeed, in the exploration process (rigorous definition is given in Definition \ref{def:rootneigh:dir}), { motivated by the PageRank problem}, we { naturally} explore directed edges {\em only in their opposite direction}. In other words, a directed edge $(j,i)$ is only explored from $i$ to $j$. Clearly, since edges are not explored in both directions, starting from the root we might not be able to explore all the graph. Heuristically, from the point of view of the root $\emp$, {\em only part of the graph has influence on the incoming neighborhood of $\emp$}. This is very different from the undirected case, where the exploration process continues until the entire graph is explored (when the graph is connected). We resolve this by introducing so-called marks to track the explored and not-explored out-edges in the graph. The precise definition of LW convergence in directed graphs is given in Section \ref{sec:LWC:dir}.


We point out that our construction is  one of many possible ways to define LW convergence for directed graphs.  For instance, Aldous and Steele~\cite{Aldous2004} allow edge weights. This might be sufficient to define an inclusion of directed graphs in the space of undirected graphs with edge weights, and use the notion for undirected graphs to define an exploration process for directed graphs. The advantage of our construction is that it requires the minimum amount of information, sufficient to prove the convergence of PageRank, which is the main problem we aim to resolve.

Definition \ref{def:LWC:dir} below, together with Remark \ref{rem:crit:Dlwc}, gives a criterion for the convergence of a sequence of directed random graphs, that can be presented as marked graphs by just assigning {\em marks equal to out-degrees}. The precise formulation requires heavy notation that we have not introduced yet, therefore we do not state it here. 

The advantage of having a  LW limit $(G,\emp)$ is that a whole family of  local properties of the graph sequence {\em can pass to the limit}, and the limit is given by a local property of $(G,\emp)$ itself. More precisely, in the construction of LW convergence, one defines a {\em distance} between (marked directed) rooted graphs (see Defintion \ref{def:locdist:und}).  Then, any function $f$ from the space of rooted graphs to $\R$ that is {\em bounded and continuous} with respect to the distance function can pass to the limit, i.e., for $V_n$ a uniformly chosen vertex in $G_n$, 
$$
	\lim_{n\rightarrow\infty} \E\left[f(G_n,V_n)\right] = \E\left[f(G,\emp)\right].
$$
This can be rather useful in understanding the asymptotic behavior of local properties of a graph sequence. As a toy example, in the undirected setting, take the function $f(G,\emp) = \I\{d_\emp = k\}$. It is easy to show, using Definition \ref{def:locdist:und}, that $f$ is a continuous function. Assume that a sequence of graphs $G_n\rightarrow (G,\emp)$ locally weakly, where  $(G,\emp)$ is random rooted graph. Then, for every $n\in\N$, 
$$
	\E\left[f(G_n,V_n)\right]  = \frac{1}{n}\sum_{i\in[n]} \pr\left(d_i = k\right) = \pr\left(d_{V_n}=k\right),
$$
i.e., $f$ evaluated on a random root is just the probability that a uniformly chosen vertex has degree $k$. As a consequence, the sequence $(G_n)_{n\in\N}$ has a limiting degree distribution given by
$$
	\lim_{n\rightarrow\infty}\pr\left(d_{V_n}=k\right) = \pr\left(d_{\emp}=k\right),
$$
where $\emp$ is the root of $G$. Other examples of continuous functions in the undirected setting are the nearest-neighbor average degree of a uniform vertex, the finite-distance neighborhood of a uniform vertex and the average pressure per particle in the Ising model. In our directed setting, it follows that, if $G_n\rightarrow (G,\emp,M(G))$,
$$
	(\mout_{V_n},\din_{V_n}) \stackrel{d}{\longrightarrow} (\mout_{\emp}, \din_{\emp}),
$$
where $M(G)$ is the set of marks of the limiting graph, $(\mout_{V_n},\din_{V_n})$ are the mark and the in-degree of a uniformly chosen vertex $V_n$, and 
$(\mout_{\emp}, \din_{\emp})$ are the mark and the in-degree of the root $\emp$ in the limiting directed graph $G$. The notation $\mout$ hints on the relation between marks and out-degrees. When marks are assigned that are equal to the out-degree, this implies the convergence of the in- and 
out-degree of a uniformly chosen vertex. One of the surprises in our version of LW convergence is that 
in the limiting graph, the mark of the root $\mout_\emp$ is not necessarily equal to the out-degree of the root.

\subsection{Application of LW convergence to PageRank}
\label{sec:PRproof_intro}

The proof of Theorem \ref{th:pagerank} is given in Section \ref{sec:pagerank}. Here we describe the   structure of the proof, explaining why the LW convergence for directed graphs is useful. Schematically, the structure of our proof of Theorem \ref{th:pagerank} is presented in Figure~\ref{fig:scheme}. The implication (A), denoted by the dashed red arrow, is the one we aim to prove. We split it in three steps (a), (b), (c), denoted by the solid black arrows. We will now explain each step.
\begin{figure}[htb]
\begin{framed}
\centering
	\includegraphics[width = 0.4\textwidth]{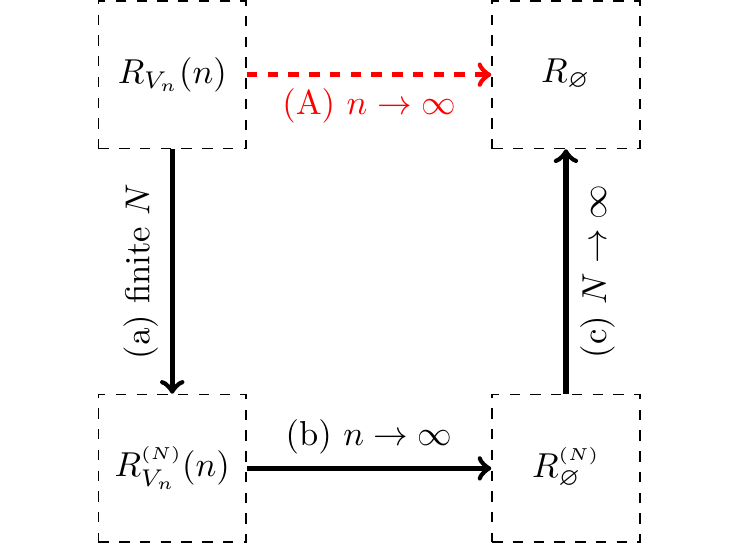}
	\caption{Structure of the proof of Theorem \ref{th:pagerank}. The (A) convergence is what we are after, the convergence in distribution of $R_{V_n}(n)$ to a limiting random variable. To prove that, we need the three different steps (a), (b), (c) given by the other arrows.}
	\label{fig:scheme}
\end{framed}
\end{figure}

\paragraph {\bfseries Step (a): Finite approximations.} It is well known \cite{Andersen,Avrac,Bianchini,Litvak2} that  PageRank can be written as
\[
	R_i(n) = (1-c)\left(1+\sum_{k=1}^\infty c^k\sum_{\ell\in \mathrm{path}_i(k)}\prod_{h=1}^k \frac{e_{\ell_{h},\ell_{h+1}}}{\dout_{\ell_h}}\right),
\]
where $\mathrm{path}_i(k)$ is the set of directed paths of $k$ steps that end at $i$. In other words,  $R_i(n)$ is a {\em weighted sum of all the directed paths that end at $i$}. In particular, we can write finite approximations for PageRank as
\[
	R^{\sss(N)}_i(n) = (1-c)\left(1+\sum_{k=1}^N c^k\sum_{\ell\in \mathrm{path}_i(k)}\prod_{h=1}^k \frac{e_{\ell_{h},\ell_{h+1}}}{\dout_{\ell_h}}\right),
\]
where now the sum is taken over all paths of length at most $N\in\N$. We use the sequence of finite approximations $\left(R^{\sss(N)}_{V_n}(n)\right)_{n\in\N}$ to estimate the PagreRank of a random vertex with {\em exponentially small precision} by its finite approximations. We prove that, for any $\varepsilon>0$, 
	$$	
		\pr\left(|R_{V_n}(n)-R_{V_n}^{\sss(N)}(n)|\geq \varepsilon\right)\leq \frac{c^{N+1}}{\varepsilon}.
	$$
		Notice that the bound is {\em independent of the graph size that we consider}. This bound is true for any directed graph of any size, so it does not require any assumption on the graph sequence.

	\paragraph{\bfseries Step (b): LW convergence.} The finite approximations of PageRank are {\em continuous} with respect to the local weak topology.  Furthermore, by definition, the $N$th approximation of PageRank depends only on the incoming neighborhood of a vertex {\em up to distance $N$}. Note that $R_i(n)$ and $R^{\sss(N)}_i(n)$ are not bounded. However, for any $r\geq0$, the function $\I\{R_{V_n}^{\sss(N)}>r \}$ is a continuous and bounded function on marked directed rooted graphs, therefore we can pass to the limit for any $N\in \N$. It follows that
	$$
		\lim_{n\rightarrow\infty}\E\left[{ \I\{R_{V_n}^{\sss(N)}>r \}}\right] = \lim_{n\rightarrow\infty}\pr\left(R^{\sss(N)}_{V_n}(n)>r\right) = \pr\left(R^{\sss(N)}_\emp >r\right),
	$$
	where in the last term $\emp$ is the root of the limiting random marked directed rooted graph $(G,\emp,M(G))$. As a consequence, every term of the sequence $(R^{\sss(N)}_{V_n}(n))_{n\in\N}$ converges in distribution. Notice that similar arguments apply for Theorem \ref{th:pagerank}(b).

	\paragraph{\bfseries Step (c): Finite approximations on the limiting graph.} On the limiting random marked directed rooted graph $(G,\emp,M(G))$, the sequence $(R^{\sss(N)}_\emp)_{N\in\N}$ is a monotonically increasing sequence of random variables. Therefore, there exists an almost sure limiting random variable $R_\emp$. Using the fact that $(G,\emp,M(G))$ is a local weak limit of a sequence of random directed graphs, and $\E[R_{V_n}]=1$ for every $n\geq1$, it is possible to prove that $\E[R_\emp]\leq 1$, so that $\pr\left(R_\emp<\infty\right)=1$.

\begin{remark}
{\rm We emphasize that the above strategy is meant just to give the intuition behind the proof. In particular, in the proof it is necessary to be careful and specify with respect to which randomness we take expectations. In fact, when we consider local weak convergence of random graphs, we have two sources of randomness: the choice of the root and the randomness of the graphs. All these are made rigorous in Section \ref{sec:pagerank}.}
\end{remark}

\subsection{Examples}
\label{sec:examples_intro}
We consider examples of directed random graphs, for which we prove LWC and find the limiting random graph. Thus, PageRank in these models converges to PageRank on the limiting graph. The following theorem makes this precise for several random graph models that have been studied in the literature. For precise definitions of the models, as well as the proof, we refer to Section \ref{sec:examples}.

\begin{Theorem}[Examples of convergence]
\label{th:meta:theorem}
The following models converge in the directed local weak sense:
\begin{enumerate}
	\item the {\em directed configuration model} converges {\em in probability};
	\item the {\em continuous-time branching processes} converge {\em almost surely};
	\item the {\em directed preferential attachment model} converges {\em in probability}.
\end{enumerate}
As a consequence, for these models there exists a limiting PageRank distribution, and the convergence holds as specified.
\end{Theorem}

\begin{Remark}[Power-law lower bound]
{\em 
The directed preferential attachment model and continuous-time branching processes both have constant out-degree. Therefore, they satisfy the condition in Remark \ref{rem:stoch:lowerbound}. Thus, their limiting PageRank distributions are stochastically bounded from below by a multiple of the limiting in-degree distributions. The directed configuration model satisfies Remark \ref{rem:stoch:lowerbound} whenever the out-degree distribution has bounded support.
}
\end{Remark}

The proof of Theorem \ref{th:meta:theorem} is divided into three propositions, respectively Proposition \ref{prop:CDM:limit} for directed configuration model, Proposition \ref{prop:CTBP:limit} for continuous-time branching processes and Proposition \ref{prop:PAM:dirLW} for directed preferential attachment model.

\subsection{Open problems}
\label{sec:open_problems}

\paragraph{Extension to exploration of outgoing edges.} In this paper we extend the definition of local weak convergence to directed graphs. Moved by the interest in PageRank algorithms on random graphs, we build our definition on the {\em exploration of incoming edges} in their opposite direction, i.e., an edge $(i,j)$ is explored from $i$ to $j$. The outgoing edges are considered as marks and we do not explore them. In the same way, it is possible to define the exploration process according to the natural direction of the edges. In this case, we consider {\em outgoing neighborhoods} instead. The definition of LW convergence would just be a consequence of symmetry. This second interpretation might be useful, for instance, in the study of diffusion processes on graphs, such as epidemic spread. An interesting and more complex extension would be to explore the incoming and outgoing neighborhoods {\em at the same time}.


%

\paragraph{PageRank on limiting graphs.}
 We are able to prove that, under relatively general assumptions, a sequence of random directed graphs admits a limiting distribution for the PageRank of a uniformly chosen vertex. In this way, we have moved the analysis of a graph's PageRank distribution from a whole sequence of graphs to a single (possibly infinite) rooted directed marked graph. Note that we prove the existence of such distribution, but we  do not always have a convenient description of it. It will be interesting to investigate the behavior of this limiting distribution. In particular, it is interesting to investigate the conditions under which the rank of the root in the limiting graph shows a {\em power-law tail}, and thus confirm the power-law hypothesis.

\vspace{0.3cm}
The remainder of the paper provides formal proofs of what has been discussed above. 

\section{Local weak convergence}
\label{sec:LWC}
\subsection{Preliminaries: LWC of undirected graphs}
\label{sec:LWC:und}
We present the definition of LWC for undirected graphs first, since the construction for directed graphs is similar. We start by defining what a rooted graph is:

\begin{Definition}[Rooted graph]
Let $G$ be a locally finite graph with vertex set $V(G)$ (finite or countable), and edge set $E(G)$. Fix a vertex $\emp\in G$ and call it the {\em root}. The pair $(G,\emp)$ is called a {\em rooted graph}.
\end{Definition}
We are not interested in the labeling of the vertices, but only in the graph structure. For this, we define isomorphisms between rooted graphs as follows:
\begin{Definition}[Isomorphism]
An isomorphism between two rooted graphs $(G,\emp)$ and $(G',\emp')$ is a bijection $\gamma : V(G) \rightarrow V(G')$ such that
\begin{enumerate}
	\item $(j,i)\in E(G)$ if and only if $(\gamma(j),\gamma(i))\in E(G)$;
	\item $\gamma(\emp) = \emp'$.
\end{enumerate}
We write $(G,\emp)\cong (G',\emp')$ to denote that $(G,\emp)$ and $(G',\emp')$ are isomorphic rooted graphs.
\end{Definition}
Denote the space of all rooted graphs (up to isomorphisms) by $\G_\star$. Formally, $\G_\star$ is the quotient space of the set of all locally finite rooted graphs with respect to the equivalence relation given by isomorphisms. 

For a rooted graph $(G, \emp)\in \G_\star$, we let $U_{\leq k}(\emp)$  denote the subgraph of $G$ of all vertices at graph
distance at most $k$ away from $\emp$. 
Formally, this means that $U_{\leq k}(\emp) = (V (U_{\leq k}(\emp)),E(U_{\leq k}(\emp)))$,
where
$$
	V(U_{\leq k}(\emp) ) = \left\{i\colon d_G(i,\emp)\leq k\right\}, \quad E(U_{\leq k}(\emp) ) = \left\{\{j,i\}\colon j,i\in  V(U_{\leq k}(\emp))\right\}.
$$
We call $U_{\leq k}(\emp) $ the $k$-neighborhood around $\emp$. We use this notion to define the distance between two rooted graphs:
\begin{Definition}[Local distance]
\label{def:locdist:und}
The function $d_{loc}((G,\emp),(G',\emp')) = 1/(1+\kappa)$, where
$$
	\kappa = \inf_{k\geq 1}\left\{U_{\leq k}(\emp) \not \cong U_{\leq k}(\emp')\right\},
$$
is called the {\em local distance} on the space of rooted graphs $\G_\star$.
\end{Definition}
It is possible to prove that $d_{loc}$ is an actual distance on the space of rooted graphs. In particular, the space $(\G_\star,d_{loc})$ is a Polish space (see \cite[Appendix A]{Gabry} for the proof for an equivalent definition of a distance).  
The function $d_{loc}$ measures how distant two rooted graphs are {\em from the point of view of the root}. In many graphs though, there is no vertex that can be naturally chosen as a root, for instance in configuration models or Erd\H{o}s-R\'enyi random graph. For this reason, it is useful to choose the root {\em at random}.
Define, for any graph $G$,
\eqn{
\label{for:probP_n}
	\P(G) = \frac{1}{n}\sum_{i\in[n]} \delta_{(G,i)}.
}
Given a graph $G$ of size $n$, $\P(G)$ is a probability measure that assigns the root uniformly at random among the $n$ vertices. When we consider a sequence of graphs $(G_n)_{n\in\N}$, we denote $\P(G_n)$ simply by $\P_n$. With this notion, we are ready to define LWC for undirected deterministic graphs:

\begin{Definition}[Local weak convergence]
\label{def:LW:determin}
Consider a deterministic sequence of locally finite graphs $(G_n)_{n\in\N}$. We say that $(G_n)_{n\in\N}$ {\em converges in the local weak sense }to a (possibly) random element $(G,\emp)$ of $\G_\star$ with law  $\P$, if, for any bounded continuous function $f:\G_\star\rightarrow\R$,
$$
	\E_{\P_n}[f]\longrightarrow \E_{\P}[f],
$$
where $\E_{\P_n}$ and $\E_{\P}$ denote the expectation with respect to $\P_n$ and $\P$, respectively.
\end{Definition}

 In particular, this means that the probability converges over open sets of the topology. Fix $(H,y)$ finite, then
\eqn{
\begin{split}
	B_R(h,y) & = \left\{(G,\emp)\in \G_\star \mbox{ : } d_{loc}((H,y),(G,\emp))\leq R\right\} \\
				 & = \left\{(G,\emp)\in \G_\star \mbox{ : } U_{\leq \lfloor 1/R\rfloor}(\emp) \cong (H,y) \right\}.
\end{split}
}
Elements in this open ball are determined by the neighborhood of the root up to distance  $ \lfloor 1/R\rfloor$. As a consequence,  the probability $\P_n$ of the ball $B_R(h,y)$ is given by 
$$
	\P_n(B_R(h,y)) = \frac{1}{n}\sum_{i\in[n]}\I\left\{U_{\leq \lfloor 1/R\rfloor}(i)\cong (H,y)\right\}.
$$
This implies that it suffices to look at the local structure of the neighborhood of a typical vertex to obtain the probability $\P_n$ of any open ball. We now state a criterion for a sequence of deterministic graphs to converge in the LW sense as in Definition \ref{def:LW:determin}:

\begin{Theorem}[Criterion for local weak convergence]
\label{th:crit:determ}
 Let $(G_n)_{n\in\N}$ be a sequence of graphs. Then $G_n$ converges in the local weak sense to $(G,\emp)$ with law $\P$ when, for every finite rooted graph $(H,y)$,
\eqn{
\label{for-critertion}
	\P_n(H) =\frac{1}{n}\sum_{i\in[n]} \I\left\{U_{\leq k}(i)\cong (H,y)\right\}\longrightarrow \P\left(U_{\leq k}(\emp)\cong (H,y)\right).
}
\end{Theorem}
The proof Theorem \ref{th:crit:determ} can be found in  \cite[Section 1.4]{vdH2}. Notice that for $(H,y)\in \G_\star$, the functions $\I\{U_{\leq k}(\emp)\cong (H,y)\}$ are continuous with respect to the local weak topology and uniquely identify the limit.

So far we have considered sequences of deterministic graphs. Whenever we consider a {\em random graph } $G_n$, we have two source of randomness. First, we have the randomness of the choice of the root, and then the randomness of the graph itself. For this reason, it is necessary to specify the randomness we take expectation with respect to, giving rise to different ways of convergence. We specify this in the following definition:

\begin{Definition}[Local weak convergence]
\label{def:LWC:und:random}
Consider a sequence of random graphs $(G_n)_{n\in\N}$, and a probability $\P$ on $\G_\star$. Denote by $\P_n$ the probability associated to $G_n$ as in \eqref{for:probP_n}. 
\begin{enumerate}
	\item We say that $G_n$ {\em converges in distribution in the local weak sense} to $\P$ if, for any bounded continuous function $f:\G_\star\rightarrow\R$,
	\eqn{
	\label{for:und:convdist}
		\E\left[\E_{\P_n}\left[f\right]\right]\longrightarrow \E_\P\left[f\right];
	}
	\item 
	We say that $G_n$ {\em converges in probability in the local weak sense} to $\P$ if, for any bounded continuous function $f:\G_\star\rightarrow\R$,
	\eqn{
	\label{for:und:convprob}
		\E_{\P_n}\left[f\right]\stackrel{\pr}{\longrightarrow} \E_\P\left[f\right];
	}
	\item 
	We say that $G_n$ {\em converges almost surely in the local weak sense} to $\P$ if, for any bounded continuous function $f:\G_\star\rightarrow\R$,
	\eqn{
	\label{for:und:conv-as}
		\E_{\P_n}\left[f\right]\stackrel{\pr-a.s.}{\longrightarrow} \E_\P\left[f\right].
	}
\end{enumerate}
\end{Definition}
Notice that the left-hand term in \eqref{for:und:convprob} is a random variable, while the right-hand side is deterministic. In fact, \eqref{for:und:convprob}  implies \eqref{for:und:convdist}, but the opposite is not true. Similarly, \eqref{for:und:conv-as} implies \eqref{for:und:convprob}.

Similarly to Theorem \ref{th:crit:determ}, we can give a criterion for the convergence of a sequence of random graphs:

\begin{Theorem}[Criterion for local weak convergence of random graphs]
\label{th:crit:rand}
 Let $(G_n)_{n\in\N}$
be a sequence random graphs. Let $(G, \emp)$ be a
random variable on $\G_\star$ having law $\P$. Then, $G_n$ converges to $(G,\emp)$ in distribution (in probability, almost surely) if \eqref{for:und:convdist} (\eqref{for:und:convprob}, \eqref{for:und:conv-as}, respectively) holds for every function of the type $\I\{U_{\leq k}(\emp)\cong (H,y)\}$, where $k\in\N$ and $(H,y)$ is a finite element of $\G_\star$.

\end{Theorem}
The proof of Theorem \ref{th:crit:rand} follows immediately from Theorem \ref{th:crit:determ}.

\subsection{Directed graphs}
\label{sec:LWC:dir}
The construction of local weak convergence for directed graphs is similar to the undirected case. It is necessary though to define an exploration process to construct the neighborhood of the root and keep track of in- and out-degrees of vertices. To keep notation as simple as possible, we use the same notation as in Section \ref{sec:LWC:und}, while here we refer to directed graphs. We start giving the definition of rooted marked directed graphs:

\begin{Definition}[Rooted marked directed graph]
\label{def:rootmarked}
Let $G$ be a directed graph with vertex set $V(G)$ and edge set $E(G)$. Let $\emp\in V(G)$ be a vertex called the {\em root}. Assume that for every $i\in V(G)$, the in-degree $\din_i$ and the out-degree $\dout_i$ of the vertex $i$ are finite.
Assign to every $i\in V(G)$ an integer value $\mout_i$ called a {\em mark}, such that $\dout_i\leq \mout_i<\infty$. Denote the set of marks by $M(G) = ( \mout_i)_{i\in V(G)}$. We call  the triplet $(G,\emp,M(G))$ a {\em rooted marked directed graph}.
\end{Definition}


To simplify notation in Definition \ref{def:rootmarked}, we will specify the marks only when necessary.
In simple words, a rooted marked directed graph is a locally finite directed graph where one of the vertices is marked as root, and to every vertex we assign a mark, which is larger than the out-degree of the vertex. If $\mout_i= \dout_i$  we keep $i$ intact, and if $\mout_i- \dout_i>0$ then we attach to $i$ exactly $\mout_i- \dout_i$ outgoing arrows pointing nowhere. This is illustrated in Figure~\ref{fig-markers}. We call a directed graph with marks, without specifying the root, a {\em marked graph}.
\begin{figure}[htb]
\begin{framed}
\centering
\includegraphics[width = 0.2 \textwidth]{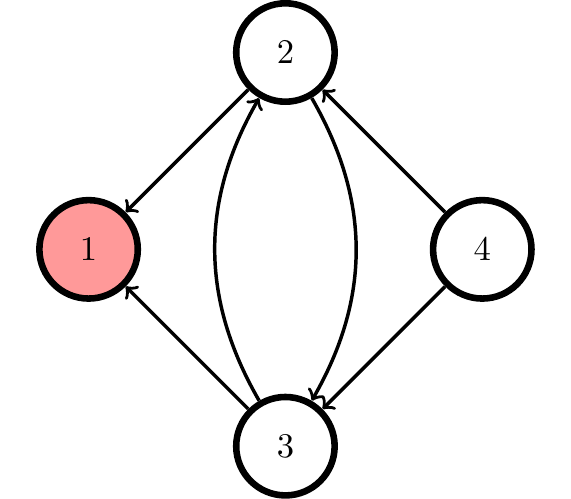}
\hspace{1cm}
\includegraphics[width = 0.2 \textwidth]{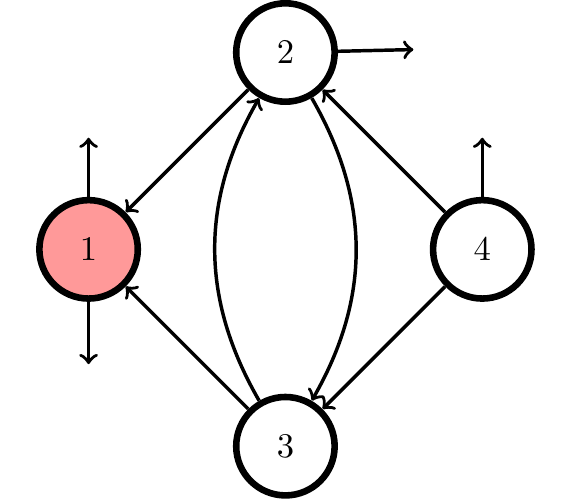}
\caption{Two examples of rooted marked directed graphs. The graph on the left is considered with marks equal to the out-degree, while in the example on the right we have assigned marks larger than the out-degree. The difference between the mark and the out-degree of a vertex can be visualized as the number of arrows starting at the vertex and pointing nowhere.}
\label{fig-markers}
\end{framed}
\end{figure}

Every directed graph can be seen as a rooted marked directed graph, with marks equal to the out-degrees and a root picked from the set of vertices. In what follows, sometimes we specify the marks, and sometimes we specify the out-degree and the number of edges pointing nowhere.

As in the undirected case, we are not interested in the precise labeling of the vertices. This leads us to define the notion of isomorphism, including the presence of marks:
\begin{Definition}[Isomorphism of rooted marked directed graphs]
\label{def:isomor:dir}
Two rooted marked directed graphs $(G,\emp,M(G))$ and $(G',\emp',M(G'))$ are {\em isomorphic} if and only if there exists a bijection $\gamma:V(G)\rightarrow V(G')$ such that
\begin{enumerate}
	\item $(i,j)\in E(G)$ if and only if $(\gamma(i),\gamma(j))\in E(G')$;
	\item $\gamma(\emp) = \emp'$;
	\item for every $i\in V(G)$, $\mout_i = \mout_{\gamma(i)}$.
\end{enumerate}
We write $(G,\emp,M(G))\cong (G',\emp',M(G'))$ to denote that $(G,\emp,M(G))$ and $(G',\emp',M(G'))$ are isomorphic rooted marked directed graphs.
\end{Definition}

Denote the space of rooted marked directed graphs by $\G_\star$, which is again a quotient space with respect to the equivalence given by isomorphisms. 
We now define the exploration process that identifies the neighborhood of the root, see Figure~\ref{fig:root:neighborhood} for an example.
\begin{figure}[htb]
\begin{framed}
\centering
\includegraphics[width = 0.25 \textwidth]{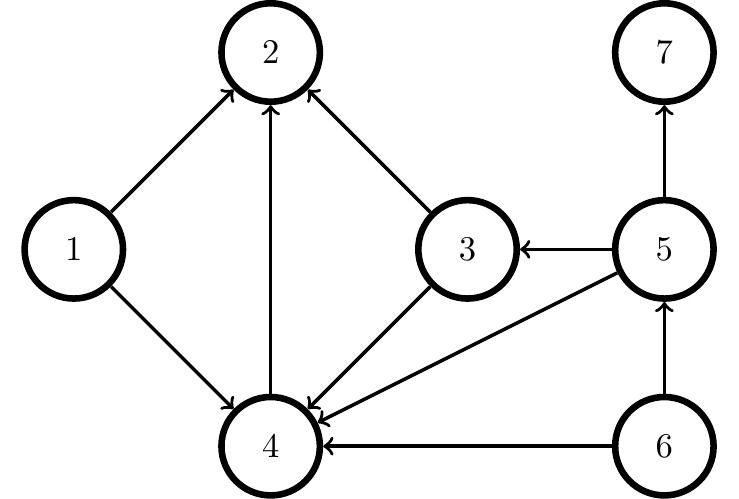} \\
\includegraphics[width = 0.2 \textwidth]{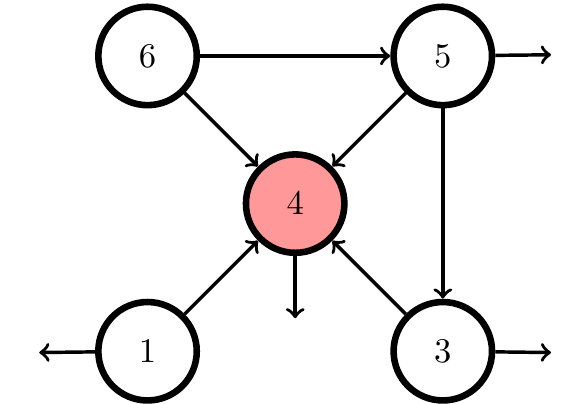}
\hspace{5cm}
\includegraphics[width = 0.2 \textwidth]{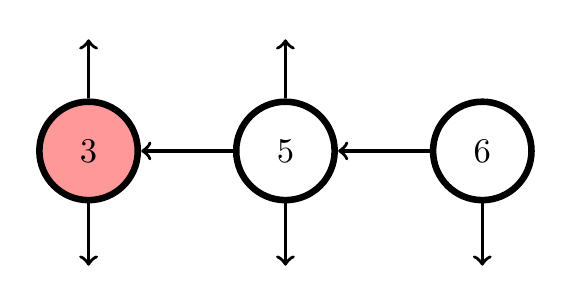}
\caption{Example of two root neighborhoods in the same graph above, where we have assigned marks equal to the out-degrees, with a different choice of the root. The root on the left is vertex 4, and vertex 3 on the right. We explore the root neighborhood up to the maximum possible distance. Notice that the graph is only partially explored in this example.}
\label{fig:root:neighborhood}
\end{framed}
\end{figure}
\begin{Definition}[Root neighborhood]
\label{def:rootneigh:dir}
Consider a rooted marked directed graph $(G,\emp,M(G))$. Fix $k\in \N$. The {\em $k$-neighborhood of  root $\emp$} is a rooted marked directed graph $(U_{\leq k}(\emp),\emp,M(U_{\leq k}(\emp)))$ constructed as follows:
\begin{itemize}
	\item[{$\rhd$}] for $k=0$, $U_{\leq  k}(\emp)$ is a graph with a single vertex $\emp$, no edges, and mark $\mout_{\emp}$;
	\item[{$\rhd$}] for $k>0$, consider $\emp$ as active, and proceed recursively as follows, for $h=1,\ldots,k$:
		\begin{enumerate}
			\item for every vertex active at step $h-1$, explore the incoming edges to the vertices in the opposite direction, finding the source of the edges;
			\item label the vertices that were active to be explored, and label the vertices just found as active, but only if they were not already found in the exploration process;
			\item for every vertex $i$ (explored or active), assign the mark  $\mout_i$ to it, that is equal to the mark in the original graph $(G,\emp,M(G))$. In addition, draw every edge between two vertices that are already found in the exploration process;
			\item if there are no more active vertices, then stop the process.
		\end{enumerate}
\end{itemize}
\end{Definition}

In this way we explore the {\em incoming neighborhood} of the root. As stated in Definition \ref{def:rootneigh:dir}, we explore edges in the opposite direction: if $(j,i)\in E(G)$ is a directed edge, then the exploration process goes from vertex $i$ to vertex $j$. Notice that it is possible that we do not explore the entire graph in this process, because we do not explore edges in all directions. This is different to the undirected case, where, for $k$ large enough, we always explore the entire graph (if connected).

We can define a local distance $d_{loc}$ on $\G_\star$ as in Definition \ref{def:locdist:und}, but this time for rooted marked directed graphs, using Definitions \ref{def:isomor:dir} and \ref{def:rootneigh:dir}. 
As in the undirected setting, the function $d_{loc}$ tells us up to what distance the neighborhoods of two roots in two different rooted marked directed graphs are isomorphic. However, in the directed setting the function $d_{loc}$ is {\it not} a metric on $\G_\star$, but it is a {\em pseudonorm}. 

Note that $d_{loc}$ is positive by definition, and obviously symmetric. It is not hard to prove that it satisfies the triangle inequality. The reason that $d_{loc}$ is not a metric is that two rooted marked directed graphs can be at distance 0 without being isomorphic. This is due to the fact that the edges can be explored only in one direction, possibly leaving parts of the graph unexplored, as mentioned above. If the {\it explorable parts} or {\em incoming neighborhoods} of two graphs from the roots are isomorphic, then the two rooted marked directed graphs are at distance zero, while these graphs still might not be isomorphic. An example is given in Figure~\ref{fig:pseudometric}.
\begin{figure}[h]
\begin{framed}
\centering
\includegraphics[width = 0.25 \textwidth]{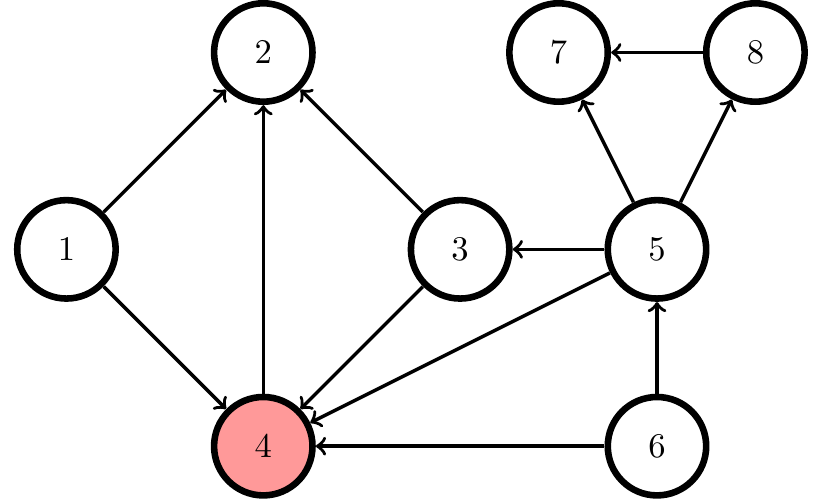}
\hspace{1cm}
\includegraphics[width = 0.25\textwidth]{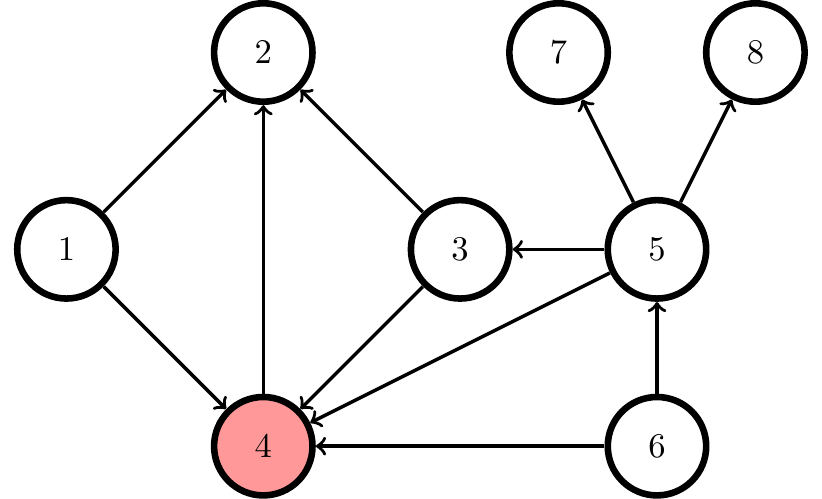}
\caption{Example of two rooted marked directed graphs that are at distance zero, but are not isomorphic. The distance between the two graphs is zero since the explorable parts of the graphs from vertex 4 (including vertices 1--6) are isomorphic, but there exists no isomorphisms between the two graphs.} 
\label{fig:pseudometric}
\end{framed}
\end{figure}
Denote the explorable neighborhood of the root by $U_\infty(\emp)$, i.e., the (possibly infinite) subgraph of a  rooted marked directed graph that can be explored from the root. Then
\eqn{
	d_{loc}((G_1,\emp_1, M(G_1)),(G_2,\emp_2, M(G_2))) = 0\quad \Longleftrightarrow \quad U_\infty(\emp_1) \cong U_\infty(\emp_2).
}
Formally, $(\G_\star,d_{loc})$ is a complete and separable space, so every Cauchy sequence has a limiting point. Although the limiting point might not be unique, the explorable neighborhood of the root is unique. The proof that the space $(\G_\star,d_{loc})$ is a complete pseudometric space is a minor adaptation of the proofs in  \cite[Appendix A]{Gabry}. 

We can define the space $\tilde{\G}_\star$ as the quotient space of $\G_\star$ using  the equivalence relation $\sim_\star$, where 
$$
	(G_1,\emp_1, M(G_1)) \sim_\star (G_2,\emp_2,M(G_2)) \quad \Leftrightarrow \quad d_{loc}\left((G_1,\emp_1,M(G_1)) , (G_2,\emp_2,M(G_2))\right)=0.
$$
On $\tilde{\G}_\star$, $d_{loc}$ is a metric. 
Any equivalence class in $\tilde{\G}_\star$ is composed by directed marked rooted graphs whose neighborhoods of the root are isomorphic. Heuristically, everything that is in the part of the graph that is not explorable from the root {\em does not have any influence on the incoming neighborhood of the root}. This means that any function on $\tilde{\G}_\star$ is well defined if and only if it is a function of the incoming neighborhood of the root.

As in the undirected sense, we denote 
\eqn{
\label{for:PG:dir}
	\P(G) = \frac{1}{|V(G)|}\sum_{i\in V(G)} \delta_{(G,i,M(G))}.
}
When we consider a sequence of marked graphs $((G_n,M(G_n)))_{n\in\N}$, we denote $\P(G_n)$  by $\P_n$. From the definition, we have that $\P(G)$ is  a probability on $\tilde{\G}_\star$, that assigns a uniformly chosen root to the marked directed finite graph. Notice that the mark set is fixed. In fact, the triplet $(G,i,M(G))$ is mapped to the equivalence class of the explorable neighborhood $U_{\infty}(i)$  of $i$ in $G$ with the same set of marks.

Since we are interested in sequences of random graphs, we give the definition of LW convergence only for random graphs:

\begin{Definition}[Local weak convergence - directed]
\label{def:LWC:dir}
Consider a sequence of random marked directed graphs $(G_n,M(G_n))_{n\in\N}$.  Let $(G,\emp,M(G))$ be a random element of $\tilde{\G}_\star$ with law $\P$. We say that $G_n$ converges {\em in distribution} ({\em in probability, almost surely}) to $\P$ if \eqref{for:und:convdist} (\eqref{for:und:convprob}, \eqref{for:und:conv-as} respectively) holds for any bounded continuous function $f:\tilde{\G}_\star\rightarrow \R$. 
\end{Definition}

\begin{Remark}[Criterion for directed LW convergence]
\label{rem:crit:Dlwc}
{\rm
The reader can observe that, once the notion of exploration process and isomorphisms in the directed case are introduced, the construction of the definition of local weak convergence for directed graphs is the same as in the undirected case. With the presence of marks we are able to keep track of the out-degrees of vertices, while we explore the incoming edges. 
	
It is easy to prove that Theorem \ref{th:crit:rand} can be extended to random marked directed graphs. In other words, it is sufficient to prove the convergence for functions of the type $\I\{U_{\leq k}(\emp)\cong (H,y,M(H))\}$, where $k\in\N$ and $(H,y,M(H))$ is a finite marked directed rooted graph.
}
\end{Remark}

\section{Convergence of PageRank}
\label{sec:pagerank}
The main result on PageRank is Theorem \ref{th:pagerank}. It states that, for a locally weakly convergent sequence of directed random graphs $(G_n)_{n\in\N}$, there exists a random variable $R_\emp$ such that the PageRank value of a uniformly chosen vertex $R_{V_n}(n)$ satisfies
$$
	R_{V_n}(n)\stackrel{d}{\displaystyle \longrightarrow} R_\emp.
$$
The random variable $R_\emp$ is defined in Proposition \ref{prop:existence:Remp} below. Notice that, even though local weak convergence is defined in terms of {\em local properties} of the graph, it is sufficient for the existence of the limiting distribution for a global property such as PageRank.

 The existence of $R_\emp$ for a sequence $(G_n)_{n\in\N}$ is assured by the convergence in {\em distribution} in the local weak sense. If $(G_n)_{n\in\N}$ converges in probability (or almost surely), then the fraction of vertices whose PageRank value exceeds a fixed value $r>0$ converges in probability (or almost surely) to a deterministic value.

\subsection{Finite approximation of PageRank}
Consider a directed graph $G_n$, and define the matrix $\sub{Q}(n)$, where $\sub{Q}(n)_{i,j} = e_{i,j}/\dout_i$, for $e_{i,j}$ the number of directed edges from $i$ to $j$.
 For $c\in(0,1]$, 
the PageRank vector $\sub{\pi}(n) = (\pi_1,\ldots,\pi_n)$ is the unique solution of
\eqn{
\label{for:rank:definition}
	\sub{\pi}(n) = \sub{\pi}(n)\left[c\sub{Q}(n)\right]+\frac{1-c}{n}\sub{1}_n \quad \mbox{ and} \quad \sum_{i=1}^n \pi_i = 1,
}
where $c\in(0,1)$ and $\sub{1}_n$ is the vector of all ones of size $n$. We are interested in the graph-normalized version of PageRank, so $\sub{R}(n) = n\sub{\pi}(n)$, which is just the PageRank vector rescaled with the size of the graph. The vector $\sub{R}(n)$ satisfies
\eqn{
\label{for:scalefree:eq}
	\sub{R}(n) = \sub{R}(n)\left[c\sub{Q}(n)\right] +(1-c)\sub{1}_{n}.
}
Denote $\Id_n$ the identity matrix of size $n$. We can solve \eqref{for:scalefree:eq} to obtain the well-known expression \cite{Andersen,Avrac,Bianchini,Litvak2}
\eqn{
\label{for:pgrnk:system}
	\sub{R}(n) = (1-c)\sub{1}_{n}\left[\Id_n-c\sub{Q}(n)\right]^{-1}.
}
In practice, the inversion operation on the matrix $\Id_n-c\sub{Q}(n)$ is inefficient, therefore, power expansion is used to approximate the matrix in \eqref{for:pgrnk:system} (see e.g. \cite{Andersen}), as 
\eqn{
\label{for:powseries}
	\left[\Id_n-c\sub{Q}(n)\right]^{-1} = \sum_{k=0}^{\infty}c^k\sub{Q}(n)^k.
}
Notice that $\sub{Q}(n)^k_{i,j}>0$ if and only if there exists a path of length exactly $k$ from $i$ to $j$, possibly with repetition of vertices and edges. Define, for $k\in\N$, 
$$
	\mathrm{path}_i(k) = \left\{\mbox{directed path $ \ell =(\ell_0, \ell_1,\ell_2, \ldots, \ell_k=i)$}\right\}.
$$ 
With this notation, we can write, for $i\in[n]$, 
\eqn{
\label{for:rank:express}
	R_i(n) = (1-c)\left(1+\sum_{k=1}^\infty c^k\sum_{\ell\in \mathrm{path}_i(n)}\prod_{h=1}^k \frac{e_{\ell_{h},\ell_{h+1}}}{d^+_{\ell_h}}\right),
}
while the $N$th finite approximation of PageRank is
\eqn{
\label{for:rank:express:N}
	R^{\sss(N)}_i(n) = \left(1-c\right)\left(1+\sum_{k=0}^N c^k\sum_{\ell\in \mathrm{path}_i(n)}\prod_{h=1}^k \frac{e_{\ell_{h},\ell_{h+1}}}{d^+_{\ell_h}}
	\right).
}
Heuristically, the PageRank formulation in \eqref{for:rank:express} includes paths of every length, while the $N$th approximation in \eqref{for:rank:express:N} discards the paths of length $N+1$ or higher. In particular, for every $i\in[n]$,  $R_i^{\sss(N)}(n)\uparrow R_i(n)$. One can write the difference between the PageRank and its finite approximation as
\eqn{
\label{for:diffRank:Ps}
	\left| R_i(n) -R^{\sss(N)}_i(n)\right| = (1-c)\sub{1}_n\sum_{k=N+1}^\infty (c \sub{Q}(n))^k_i.
}
We can prove that we can approximate the PageRank value of a randomly chosen vertex by a finite approximation with an exponentially small error, that is independent of the size of the graph:
\begin{Lemma}[Finite iterations]
\label{lem:finite:iteration}
Consider a directed graph $G_n$ and denote a uniformly chosen vertex by $V_n$. Then,
$$
	\E\left[ R_{V_n}(n) -R^{\sss(N)}_{V_n}(n)\right] \leq c^{N+1},
$$
where the bound is independent of $n$.
\end{Lemma}
\begin{proof}
Consider \eqref{for:diffRank:Ps} for a uniformly chosen vertex. We have 
\eqn{
\label{for:lemfinite:1}
	\E\left[R_{V_n}(n) -R^{\sss(N)}_{V_n}(n) \right]= \frac{1-c}{n} \sum_{i=1}^n\sum_{k=N+1}^\infty \left[\sub{1}_n(c \sub{Q}(n))^k\right]_i.
}
We write $\sub{Q}(n)^k_{j,i}$ to denote the element $(j,i)$ of the matrix $\sub{Q}(n)^k$. We write
\eqn{
\label{for:lemfinite:2}
	\left[\sub{1}_n(c\sub{Q}(n))^k\right]_i = c^k\sum_{j=1}^n  \sub{Q}(n)^k_{j,i}.
}
Substituting \eqref{for:lemfinite:2} in \eqref{for:lemfinite:1}, we obtain
$$
	\E\left[R_{V_n}(n) -R^{\sss(N)}_{V_n}(n)\right] = (1-c)\sum_{k=N+1}^\infty c^k \frac{1}{n}\sum_{i,j} \sub{Q}(n)^k_{j,i}.
$$
Since $\sub{Q}(n)^k$ is a (sub)stochastic matrix, 
$$
	\sum_{i=1}^n \sub{Q}(n)^k_{j,i} \leq  1
$$
for every $j\in[n]$. It follows that 
$$
	\E\left[R_{V_n}(n) -R^{\sss(N)}_{V_n}(n)\right] \leq (1-c)\sum_{k=N+1}^\infty c^k \frac{1}{n}\sum_{i=1}^n1 = (1-c)\sum_{k=N+1}^\infty c^k = c^{N+1}.
$$
\end{proof}
Lemma \ref{lem:finite:iteration} means that we can approximate the PageRank value of a uniformly chosen vertex with an arbitrary precision in a finite number of iterations, that is independent of the graph size. This is the starting point of our analysis. 

\subsection{PageRank on marked directed graphs}
In this section we show how the graph-normalized version of PageRank of a uniformly chosen vertex in a sequence of directed graphs $(G_n)_{n\in\N}$ admits a limiting distribution whenever $G_n$ converges in the local weak sense to a distribution $\P$. The advantage is that such a limiting distribution is expressed in terms of functions of $\P$.

The first step is to write PageRank as functions of marked directed rooted graphs that are {\em bounded and continuous} with respect to the topology given by $d_{loc}$. In this way, by the definition of local weak convergence, we can pass to the limit and find the limiting distribution.

Fix $n\in\N$. Consider a marked rooted directed graph $(G,\emp, M(G))\in \G_\star$ of size $n$. Denote as before, for $k\in\N$, 
$$
	\mathrm{path}_\emp(k) = \left\{\mbox{directed paths $ \ell =(\ell_0, \ell_1,\ell_2, \ldots, \ell_k=\emp)$}\right\},
$$
i.e., the set of directed paths in $(G,\emp,M(G))$ of length exactly $k+1$ whose endpoint is the root $\emp$. It is clear that this set is completely determined by $U_{\leq k}(\emp)$ in $(G,\emp,M(G))$. 

Consider a directed marked graph $(G_n, M(G_n))$, where we consider marks equal to the out-degrees. We have that
\eqn{
\label{for:Nrank:rooted}
\begin{split}
	R^{\sss(N)}_{V_n}(n)
	& = \sum_{i\in[n]} \I_{\{V_n = i\}}(1-c)\left(1+\sum_{k=1}^N \sum_{\pi\in \mathrm{path}_i(k)} \prod_{h=1}^k c\frac{e_{\pi_{h},\pi_{h+1}}}{\dout_{\pi_h}}\right)\\
	& =: R^{\sss(N)}[(G_n,V_n,M(G_n))],
\end{split}
}
where the last term in \eqref{for:Nrank:rooted} is a function of a marked rooted graph, evaluated on $(G_n,V_n,M(G_n))$, with $V_n$ a uniformly chosen root. In particular, we can see the $N$th approximation of PageRank as a function of the marked rooted graph.  We call the function $ R^{\sss(N)}:\tilde{\G}_\star\rightarrow \R$ the {\em root $N$-PageRank}.

Clearly, the root $N$-PageRank $R^{\sss(N)}$ is a function of $U_{\leq N}(\emp)$ only. It depends in fact on the vertices, edges and marks that are considered when exploring the graph from the root up to distance $N$.  Notice that, since the dependence on the marked  directed rooted graph is given only by $U_{\leq k}(\emp)$, the function $R^{\sss(N)}$ is well defined on any equivalence class in $\tilde{\G}_\star$.

In addition, the function $R^{\sss(N)}$ is {\em continuous} with respect to the topology generated by $d_{loc}$. In fact, since $R^{\sss(N)}$ depends only on the root neighborhood up to distance $N$, whenever two elements $(G,\emp,M(G))$ and $(G',\emp',M(G'))$ are at distance less than $1/(1+N)$, their roots neighborhoods are isomorphic up to distance $N+1$, which implies that $R^{\sss(N)}[(G,\emp,M(G))] = R^{\sss(N)}[(G',\emp',M(G'))]$. 

The problem is that $R^{\sss(N)}$ {\em is not bounded}, so LWC does not assure that we can pass to the limit. To resolve this, we introduce a different type of function:

\begin{Definition}[Root $N$-PageRank tail]
\label{def:root:rank:tail}
Fix $N\in\N$. For $r>0$, define $\Psi_{r,N} :\tilde{\G}_\star\rightarrow\{0,1\}$ by 
$$
	\Psi_{r,N} \left[(G,\emp,M(G))\right] := \I\left\{ R^{\sss(N)}\left[(G,\emp,M(G))\right]>r\right\}.
$$
We call the function $\Psi_{r,N}$ the {\em root $N$-PageRank tail} at $r$.
\end{Definition}

The function $\Psi_{r,N}$ is clearly bounded, and  it depends only on the neighborhood of the root $\emp$ up to distance $N$ through the function $R^{\sss(N)}$. This means that, for any $r>0$, $\Psi_{r,N}$ is continuous. 

Since the root $N$-PageRank on $\tilde{\G}_\star$ represents the $N$th approximation of PageRank on directed graphs, it follows that
\eqn{
\label{for:pagerank:rooted:1}
	\E_{\P_n}[\Psi_{r,N}]= \frac{1}{n}\sum_{i\in[n]}\I\left\{R^{\sss(N)}_{i}(n)>r\right\},
}
i.e., $\E_{\P_n}[\Psi_{r,N}]$ is the empirical fraction of vertices in $G$ such that the $N$th approximation of PageRank exceeds $r$.
In particular, for every $r\geq0$, if $G_n\rightarrow \P$ in distribution, 
\eqn{
\label{for:con:dist:rank}
	\pr\left(R^{\sss(N)}_{V_n}(n)>r\right) = \E\left[\frac{1}{n}\sum_{i\in[n]}\I\left\{R^{\sss(N)}_{i}(n)>r\right\}\right]  \longrightarrow \P\left(R^{\sss(N)}_\emp \geq r\right),
}
while for convergence in probability (or almost surely), the limit in \eqref{for:con:dist:rank} exists in probability (or almost surely). Consider the sequence of random variables $(R^{\sss(N)}_\emp)_{N\in\N}$, where
$$
	R^{\sss(N)}_\emp := R^{\sss(N)}[(G,\emp,M(G))],
$$
where $(G,\emp,M(G))$ is a random directed rooted graph with law $\P$. From \eqref{for:con:dist:rank}, it follows that $R^{\sss(N)}_{V_n}(n)\rightarrow R^{\sss(N)}_\emp$ in distribution.

We have just proved that, for a sequence of directed graphs $(G_n)_{n\in\N}$ that converges locally weakly to $\P$, any finite approximation of the PageRank value of a uniformly chosen vertex converges in distribution to a limiting random variable, which is given by a function of $\P$. 

\subsection{The limit of finite root ranks}
Assume that the sequence $(G_n)_{n\in\N}$ of directed graphs converges to a directed rooted marked graph $(G,\emp,M(G))$ with law $\P$. In principle, such limiting $(G,\emp,M(G))$ can be an infinite directed rooted marked graph. Because of this, we cannot simply take the limit as $N\rightarrow\infty$ of the sequence $(R^{\sss(N)}_\emp)_{N\in\N}$, where $\emp$ is the root of $(G,\emp,M(G))$, because the PageRank algorithm is not defined on an infinite graph. Nevertheless, if $\P$ is a LW limit of some sequence of directed graphs, it admits a such limit:

\begin{Proposition}[Existence of limiting root rank]
\label{prop:existence:Remp}
Let $\P$ be a probability on $\tilde{\G}_\star$. If $\P$ is the LW limit in distribution of a sequence of marked directed graphs $(G_n)_{n\in\N}$,  then there exists a random variable $R_\emp$ with $\E_{\P}[R_\emp]\leq 1$, such that  $\P$-a.s. $R^{\sss(N)}_\emp\rightarrow R_\emp$. As a consequence, $\P(R_\emp<\infty)=1$.
\end{Proposition}
\begin{proof}
Clearly, the sequence $(R^{\sss(N)}_\emp)_{N\in\N}$ is $\P$-a.s. increasing. Therefore,  the almost sure limit $R_\emp = \lim_{N\rightarrow\infty}R^{\sss(N)}_\emp$  exists. This is independent of the fact that $\P$ is a LW limit. 

By  LW convergence, we know that $R^{\sss(N)}_{V_n}(n)\rightarrow R^{\sss(N)}_\emp$ in distribution. For every $N\in\N$, by Fatou's Lemma we can bound
$$
	\E_{\P}\left[R^{\sss(N)}_\emp\right] \leq \liminf_{n\in\N}\E\left[R^{\sss(N)}_{V_n}(n)\right]\leq \liminf_{n\in\N}\E\left[R_{V_n}(n)\right] = 1,
$$
where the second bound comes from the fact that any $N$-finite approximation of PageRank is less than the actual PageRank value, and the fact that the graph-normalized PageRank has expected value 1.  Since $(R^{\sss(N)}_\emp)_{N\in\N}$  is increasing, we conclude that there exists $z\leq 1$ such that
$$
	\E_\P\left[R_\emp\right] = \lim_{N\rightarrow \infty }\E_\P\left[R^{\sss(N)}_\emp\right] =z.
$$
\end{proof}

\subsection{Proof of Theorem \ref{th:pagerank}}
We start with implication (a) of Theorem \ref{th:pagerank}. We want to prove that $R_{V_n}(n)$ converges to $R_\emp$ in distribution. So, for every $r\geq0$ and $\varepsilon>0$ there exists $M(\varepsilon)\in\N$ such that, for every $n\geq M(\varepsilon)$, 
\eqn{
\label{for:pgrnk:conv:1}
	\left|\pr\left(R_{V_n}(n)>r\right)-\P\left(R_\emp>r\right)\right|\leq \varepsilon.
}
We can write, using the triangle inequality,
\eqn{
\label{for:pgrnk:conv:2}
\begin{split}
	\left|\pr\left(R_{V_n}(n)>r\right)-\P\left(R_\emp>r\right)\right| \leq &
		\left|\pr\left(R_{V_n}(n)>r\right)-\E\left[\P_n\left(R^{\sss(N)}_\emp>r\right)\right]\right|\\
		& + \left|\E\left[\P_n\left(R^{\sss(N)}_\emp>r\right)\right]-\P\left(R^{\sss(N)}_\emp>r\right)\right| \\
		& + \left|\P\left(R^{\sss(N)}_\emp>r\right)-\P\left(R_\emp>r\right)\right|.
\end{split}
}
We show that \eqref{for:pgrnk:conv:1} holds by proving that every term in the left hand side of \eqref{for:pgrnk:conv:2} can be bounded by $\varepsilon/3$.  

By Lemma \ref{lem:finite:iteration} we can bound the first term with $c^{N+1}$ (independently of $n$). Therefore, defining $N_1 = \log_c(\varepsilon/3)$  and taking $N>N_1$, the first term is bounded by $\varepsilon/3$. 

For the last term, we apply Proposition~\ref{prop:existence:Remp}, so we can find $N_2 = N_2(\varepsilon)\in\N$ such that, for every $N\geq N_2$,
$$
	\left|\P\left(R^{\sss(N)}_\emp>r\right)-\P\left(R_\emp>r\right)\right|\leq \varepsilon/3.
$$
Set $N_0(\varepsilon) = \max(N_1,N_2)$. For any $N\geq N_0$, both the first and third terms are bounded by $\varepsilon/3$.
Using LW convergence in distribution, we can find $M(N_0,\varepsilon)\in\N$ such that, for every $n\geq M$, the second term is bounded by $\varepsilon/3$. This completes the proof of statement (a).

For statement (b),  we need to show that, for every $r>0$, as $n\rightarrow\infty$,
$$
	\frac{1}{n}\sum_{i=1}^n \I\{R_i(n)>r\}\stackrel{\pr}{\longrightarrow}\P\left(R_\emp>r\right).
$$ 
For every $N\in\N\cup\{\infty\}$ and $r\geq 0$, we denote the empirical fraction of vertices whose $N$th approximation of PageRank in $G_n$ exceeds $r$ by
$$
	\bar{R}(n;r,N) := \frac{1}{n}\sum_{i=1}^n \I\{R^{\sss(N)}_i(n)>r\}.
$$
If $N=\infty$, then $\bar{R}(n;r,N) = \bar{R}(n;r)$ is the empirical tail distribution of PageRank. By  LW convergence in probability of $(G_n)_{n\in\N}$, we know that, for every $N\in\N$ and $r>0$, 
\eqn{
\label{for:pgrnk:conv:3}
	\bar{R}(n;r,N)\stackrel{\pr}{\longrightarrow}\P\left(R_\emp^{(N)}>r\right).
} 
Fix $r>0$, $\varepsilon>0$. We need to show that for every $\delta>0$ there exists $n_0(\delta)\in\N$ such that, for any $n\geq n_0$, $\pr\left(\left|\bar{R}(n;r)-\P(R_\emp>r)\right|\geq \varepsilon\right)\leq\delta$. We can write, for $N$ to be fixed,
\eqn{
\label{for:pgrnk:conv:4}
\begin{split}
	\pr\left(\left|\bar{R}(n;r)-\P(R_\emp>r)\right|\geq \varepsilon\right)
	\leq &
		\frac{1}{\varepsilon}\bigg[\E[\bar{R}(n;r)-\bar{R}(n;r,N)]\\
		&+\E[|\bar{R}(n;r,N)-\P(R^{\sss(N)}_\emp>r)|]\\
		&+
			|\P(R^{\sss(N)}_\emp>r)-\P(R_\emp>r)|\bigg].
\end{split}
}
Similarly to \eqref{for:pgrnk:conv:2}, we can find $n$ and $N$ large enough such that every term in the right-hand side of \eqref{for:pgrnk:conv:4} is less than $\delta\varepsilon/3$. 

For the first term, we apply Lemma \ref{lem:finite:iteration}, so we can find $N_1$ large enough such that $c^{N_1+1}\leq  \delta\varepsilon/3$. For the last term, we apply Proposition \ref{prop:existence:Remp}, so we can find $N_2$ such that the last term is less than $\delta\varepsilon/3$. 

Take $N_0 = \max\{N_1,N_2\}$. Then, by \eqref{for:pgrnk:conv:3} and the fact that $\{\bar{R}(n;r,N)\}_{n\in\N}$ is uniformly integrable (since $\bar{R}(n;r,N)\leq1$), we can find $n_0$ big enough such that 
$$
	\E[|\bar{R}(n;r,N)-\P(R^{\sss(N)}_\emp>r)|]\leq \delta\varepsilon/3
$$
for all $n>n_0$, $N>N_0$. As a consequence, we conclude that, for any $n\geq n_0$, 
$$
	\pr\left(\left|\bar{R}(n;r)-\P(R_\emp>r)\right|\geq \varepsilon\right)\leq \delta,
$$
which proves the convergence in probability. 

\subsection{Undirected graphs} Undirected graphs are in fact a special case of directed graphs, where each link is reciprocated. Theorem~\ref{th:pagerank} does not make any assumption concerning link reciprocation, and thus it simply holds for undirected graphs as well. In that case, we may use the standard notion of the LWC for undirected graphs,  as described in Section~\ref{sec:LWC:und}, and it is not hard to see that our notion of directed LW convergence reduces to this.

Let us explain why the special case of undirected graphs deserves our attention. Indeed, usually, undirected graphs are easier to analyze than directed ones. For example, the adjacency matrix of an undirected graph is symmetric, which implies many nice properties. However, PageRank is based on directed paths, and its analysis is greatly simplified when these paths do not contain cycles, with high probability. 

For example, PageRank can be written as a product of three terms, one of which is the expected number of visits to $i$, starting from $i$, by a simple random walk, which terminates at each step with probability $c$~\cite{Avrachenkov2006effect}.  Now notice that in undirected graphs, each edge can be traversed in both directions, hence, a path starting at $i$ may return to $i$ in only two steps, so the average number of visits to $i$ will be  a random variable that depends on the entire neighborhood. In contrast, e.g., in the directed configuration model, returning to $i$ is highly unlikely. This makes PageRank in undirected graphs hard to analyze, and only few results have been obtained so far (see e.g. \cite{Avrachenkov2015undirected}). 

Our result simultaneously covers the directed and the undirected cases because we only state the equivalence between the behavior of PageRank on a graph and on its limiting object. In this setting, the difficulties that arise in the analysis of PageRank on undirected graphs are,  in fact, `postponed' to the (undirected) limiting random graph.

\section{Generalized PageRank} 
\label{sec:generalized}

\subsection{Universality of finite approximations}
In this section we will show that Theorem~\ref{th:pagerank} extends to generalized PageRank as given in \eqref{for:recursive:pagerank:pers:norm}. We will assume that $C_j\le c<1$, $j\in[n]$ are bounded away from one and that the vector ${\bf B}_n=(B_i)_{i\in[n]}$ consists of i.i.d.\ random variables that are independent of the graph $G_n$, and we let $\E(B_1)=1-c$ to keep the argument close to the basic case.

In this generalized setting, the proof of Lemma~\ref{lem:finite:iteration} goes through almost without changes. Let $\sub{A}$ be a matrix such that $A_{ij}(n)=C_j {e}_{ji}/\Dout_j$. Recall that $Q_{ij}(n)={e}_{ji}/\Dout_j$. Since $C_i\le c<1$ holds for all $i\in[n]$,
\begin{align*}
	\E\left[R_{V_n}(n) -R^{\sss(N)}_{V_n}(n) \right]&=\E\left[\sum_{i=1}^n\sum_{k=N+1}^\infty \left[{\bf B}_n(\sub{A}(n))^k\right]_i\right]\le \E\left[\sum_{i=1}^n\sum_{k=N+1}^\infty c^k \left[{\bf B}_n\sub{Q}(n)^k\right]_i\right]\\
	& = \E\left[\sum_{i=1}^n\sum_{k=N+1}^\infty c^k\sum_{j=1}^n B_j \sub{Q}(n)^k_{j,i}\right]\\
	&= \sum_{j=1}^n\sum_{k=N+1}^\infty c^k\,(1-c)\,\sum_{i=1}^n\sub{Q}(n)^k_{j,i}\le c^{N+1},
\end{align*}
where in the final equality we have used the independence of $B_j$ and the graph $G_n$ (and thus $\sub{Q}(n)$).

Furthermore, Proposition~\ref{prop:existence:Remp} goes through without changes. The only difference is that additional randomness arises through the random $(C_i)_{i\in[n]}$ and $(B_i)_{i\in[n]}$. Therefore, for generalized PageRank, the first and the last terms in  \eqref{for:pgrnk:conv:2} and \eqref{for:pgrnk:conv:4} can be bounded exactly as before. This is natural because the first and the last terms approximate the PageRank in, respectively, original graph and the limiting graph, by finite iterations, and this approximation does not depend on the random $(C_i)_{i\in[n]}$ and $(B_i)_{i\in[n]}$ under quite general assumptions. 


It remains to analyze the second term in \eqref{for:pgrnk:conv:2} and \eqref{for:pgrnk:conv:4}. This is more tricky because this term bounds the difference between the finite random graph and the limiting object. Difficulties arise since $(C_i)_{i\in[n]}$ and $(B_i)_{i\in[n]}$ are associated to vertex labels in $[n]$. This information is lost in the LW limit, therefore additional assumptions are necessary to prove that the second term in \eqref{for:pgrnk:conv:2} and \eqref{for:pgrnk:conv:4} is small.  We next discuss two possible settings how LWC can be used in the generalized PageRank setting. 

\subsection{Independent $(C_i)_{i\in[n]}$ and $(B_i)_{i\in[n]}$}
\label{sec:LWC:CB:indep}
First, we assume that $(C_i)_{i\in[n]}$ and $(B_i)_{i\in[n]}$ are independent of the graph sequence $(G_n)_{n\in\N}$, and $(C_i)_{i\in[n]}$ and $(B_i)_{i\in[n]}$ are each i.i.d.\ sequences that are independent of each other. In this case, on the limiting marked rooted graph $(G,\emp,M(G))$ we assign to every vertex $v\in V(G)$ independent samples $C_v$ and $B_v$. In this case, for $(H,y,M(H))$ a finite marked rooted graph, since $(C_i)_{i\in[n]}$ and $(B_i)_{i\in[n]}$ are independent of the graph, 
\eqn{
\label{for:generalized:cor:1}
\begin{split}
	&\frac{1}{n}\sum_{i\in[n]}\pr\left(R^{\sss(N)}_{i}(n)>r~|~U_{\leq N}(i)\cong (H,y,M(H))\right)\pr\left(U_{\leq N}(i)\cong (H,y,M(H))\right)\\
& = \pr(\widehat{R}^{\sss(N)}(H,y,M(H))>r)\frac{1}{n}\sum_{i\in[n]}\pr\left(U_{\leq N}(i)\cong (H,y,M(H))\right),
\end{split}
} 
where now $\I\{\widehat{R}^{\sss(N)}(H,y,M(H))>r\}$ is a function of the finite structure given by $(H,y,M(H))$, where the randomness is only given by a finite number of $(C_i)_{i\in[n]}$ and $(B_i)_{i\in[n]}$.  We note that \eqref{for:generalized:cor:1} only assumes that $(C_i)_{i\in[n]}$ and $(B_i)_{i\in[n]}$ are independent of the graph sequence $(G_n)_{n\in\N}$. In order to be able to define the local-weak limit, though, we further need the independence and i.i.d.\ assumptions on $(C_i)_{i\in[n]}$ and $(B_i)_{i\in[n]}$. Then, a similar expression to \eqref{for:generalized:cor:1} holds for the limiting graph $(G,\emp,M(G))$. As a consequence, the second term in \eqref{for:pgrnk:conv:2} can be written as
\eqn{
\label{for:generalized:cor:2}
\begin{split}
	& \left|\E\left[\P_n\left(R^{\sss(N)}_\emp>r\right)\right]-\P\left(R^{\sss(N)}_\emp>r\right)\right| \\
	& = \sum_{\sss(H,y,M(H))}\pr(\widehat{R}^{\sss(N)}(H,y,M(H))>r)\\
	&\quad \quad   \times \bigg|\E[\P_n(U_{\leq N}(\emp)\cong (H,y,M(H)))]-\P(U_{\leq N}(\emp)(H,y,M(H)))\bigg|\\
	&\leq \sum_{\sss(H,y,M(H))}\bigg|\E[\P_n(U_{\leq N}(\emp)\cong (H,y,M(H)))]-\P(U_{\leq N}(\emp)(H,y,M(H)))\bigg|\\
	& = 2d_{\sss \mathrm{TV}}(\E\P_n,\P),
\end{split}
}
where $\E\P_n$ is the distribution given by $\E[\P_n(\cdot)]$, and the last term is the total variation (TV) distance between $\P$ and $\E\P_n$. Since $\tilde{\G}_\star$ is discrete, convergence in distribution implies convergence in TV distance, so that $ 2d_{\sss\mathrm{TV}}(\E\P_n,\P)=o(1)$. The fact that the term $\pr(\widehat{R}^{\sss(N)}(H,y,M(H))>r)$ is the same for the graph sequence and the limit comes from the fact we are looking at expectations of i.i.d.\ random variables on a given structure $(H,y,M(H))$. 

The bound in \eqref{for:generalized:cor:2} is enough to conclude that the generalized PageRank with $(C_i)_{i\in[n]}$ and $(B_i)_{i\in[n]}$ independent of the graph, and themselves independent i.i.d.\ sequences, converges in distribution. Here no further assumptions are made on the distributions $C$ and $B$. Such result does not apply to the convergence in probability, since \eqref{for:generalized:cor:1} is an expectation with respect to the random graph.

In this setting, for $N\in\N$, the limiting distribution $R^{\sss (N)}_\emp$ of the $N$th approximation of PageRank is again a weighted sum of all paths of length at most $N$ that ends at the root $\emp$. In particular, $R^{\sss (N)}_\emp$  is given by
$$
	R^{\sss (N)}_\emp = \sum_{k=0}^N\sum_{\ell\in \mathrm{path}_\emp(k)}B_{\ell_{k}}\prod_{h=1}^k \frac{C_{\ell_h}}{\mout_{\ell_h}}.
$$
where now a path $\ell\in \mathrm{path}_\emp(k)$ contributes with the weight $B_{\ell_{k}}\prod_{h=1}^k C_{\ell_h}/\mout_{\ell_h}$, and again, all the appearing $(C_i)_{i\geq 1}$ and $(B_i)_{i\geq 1}$ are independent i.i.d.\ sequences.

\subsection{Extended directed LW convergence}
\label{sec:LWC:CB:extended}
The advantage of \eqref{for:generalized:cor:1} is that, once the structure $(H,y,M(H))$ is fixed, the probability that PageRank exceeds $r$ is given by an expectation in terms of $(C_i)_{i\in[n]}$ and $(B_i)_{i\in[n]}$.  Equation \eqref{for:generalized:cor:1} does not extend to convergence in probability, since we are taking expectations. In fact, when considering convergence in probability, we have to prove that the second term in \eqref{for:pgrnk:conv:4} converges to zero in probability. With a similar argument as the one that we have used to get \eqref{for:generalized:cor:1}, for any $(H,y,M(H))$ finite marked directed rooted graph,
	\eqn{
	\frac{1}{n}\sum_{i\in[n]}\I\{R^{\sss(N)}_{i}(n)>r,\ U_{\leq N}(i)\cong (H,y,M(H))\}-\P\left(R^{\sss(N)}_\emp>r,\ U_{\leq N}(\emp)\cong (H,y,M(H))\right).
	}
Here, the convergence in probability of the graph sequence is not enough to conclude that the sum over all possible finite structures $(H,y,M(H))$ is small. 

In order to prove this convergence in probability, we need to include $(C_i)_{i\in[n]}$ and $(B_i)_{i\in[n]}$ as additional marks in the definition of directed marked rooted graphs. In the exploration process described in Definition \ref{def:rootneigh:dir}, to every explored vertex $v$ we assign a mark $\mout_v$ that is equal to the mark of $v$ in the starting graph. Assuming that $(C_i)_{i\in[n]}$ and $(B_i)_{i\in[n]}$ take {\em discrete values}, we can assign a multi-mark $(\mout_v,C_v,B_v)$ to vertices found in the exploration process. Here, we then need no independence assumptions on $(C_i)_{i\in[n]}$ and $(B_i)_{i\in[n]}$ w.r.t.\ the graph $G_n$, but beware that the notion of multi-marked LWC has become significantly stronger.

This leads to an extended definition of local weak convergence on directed multi-marked rooted graphs, where now the definition of isomorphism (as in Definition \ref{def:isomor:dir}) includes the preservation of the multi-marks. More precisely, an isomorphisms between two directed multi-marked rooted graphs $(G,\emp,M(G))$ and $(G',\emp',M(G'))$ is a map $\gamma:G\rightarrow G'$  such that it satisfies Definition \ref{def:isomor:dir} and , for every $v\in V(G)$, $C_{\gamma(v)} = C_v$ and $B_{\gamma(v)} = B_v$.

It is easy to verify that the construction of the extended directed local weak convergence is the same as the one presented in Section \ref{sec:LWC:dir}, where now instead of marks we consider multi-marks. As a consequence, the family of functions $(\I\{R^{\sss (N)}_\emp>r\})_{N\in\N}$ is continuous with respect to the topology of the extended directed LW convergence, therefore \eqref{for:pgrnk:conv:4} follows immediately. In the next section, we formalize these two different approaches to LWC.

\subsection{Formulation of the result for generalized PageRank}
We can summarize the results discussed for the generalized PageRank in the following theorem:

\begin{Theorem}[Asymptotic generalized PageRank distribution]
\label{th:pagerank:pers}
Let $(G_n)_{n\in\N}$ be a sequence of directed random graphs. Consider the generalized PageRank as in \eqref{for:recursive:pagerank:pers:norm}, where, for $j\in[n]$,  $A_j=C_j/\Dout_j$, where $C_j$'s are random variables bounded by $c<1$ and the random vector $(B_i)_{i\in[n]}$ satisfies $\E(B_1)=1-c$ and is independent of $G_n$. Then, the following holds:
\begin{enumerate}
	\item[{(a)}] Assume that $(C_i)_{i\in[n]}$ are i.i.d., $(C_i)_{i\in[n]}$ is independent of  $(B_i)_{i\in[n]}$,$(C_i)_{i\in[n]}$ and $(B_i)_{i\in[n]}$ are independent i.i.d.\ sequences that are independent of $G_n$. If $G_n$ converges LW in distribution in the sense of Definition \ref{def:LWC:dir}, then there exists a distribution $R_\emp$ such that $R_{V_n}(n)\stackrel{d}{\rightarrow}R_\emp$;
	\item[{(b)}] Assume that $(C_i)_{i\in[n]}$ and $(B_i)_{i\in[n]}$ take discrete values. Then, Theorem \ref{th:pagerank} holds for the extended LWC for multi-marked directed graphs defined in Section \ref{sec:LWC:CB:extended}.
\end{enumerate}
\end{Theorem}

Theorem \ref{th:pagerank:pers}(a) is given by the independent setting in Section \ref{sec:LWC:CB:indep}. This method is simpler, in the sense that it does not require additional constructions than the ones used to prove Theorem \ref{th:pagerank}. On the other hand, it gives a weaker result, since the convergence holds in distribution. Also, we need to assume that $(C_i)_{i\in[n]}$ are i.i.d.\ and they are independent of $(B_i)_{i\in[n]}$ and the graph $G_n$. In this case, it is not clear what the appropriate conditions are under which LWC in probability holds.

Theorem \ref{th:pagerank:pers}(b) depends on the extended LWC notion of Section \ref{sec:LWC:CB:extended}. The reformulation of LWC requires less assumptions, in the sense that now we allow the distribution $(\Din,\Dout,C,B)$ to have dependent components. The disadvantage is that, to incorporate $(C_i)_{i\in[n]}$ and $(B_i)_{i\in[n]}$ in the definition of isomorphism, we require them to take discrete values, and the notion of LWC is stronger. This might not be suitable for applications. We next remark about a possible way to avoid this unnatural discreteness assumption:

\begin{Remark}[Weighted rooted graphs]
{\em 
	Benjamini, Lyons  and Schramm \cite{benjamini15} consider undirected LWC in the case of weighted edges. In particular, they define a different metric on the space of weighted rooted graphs, that includes the distance between edge weights. This construction can be extended to vertex weights, and it would lead to a different approach to investigate generalized PageRank. This requires additional work, for example due to the fact that the metric in \cite{benjamini15} is not a simple extension of the metrics that we consider in Sections \ref{sec:LWC:und} and \ref{sec:LWC:dir}. We refrain from studying this further.
}
\end{Remark}

\section{Examples of directed local weak convergence}
\label{sec:examples}
\subsection{Directed configuration model}
\label{sec:ex:DCM}

The directed configuration model (DCM) is a version of the configuration model where half-edges are labeled as in- and out-half-edges. In this setting, $\mathrm{DCM}_n$ is a directed graph of size $n\in\N$ with prescribed in- and out- degree sequences. We denote the in-degree sequence by $\subDin_n = (\Din_1,\ldots,\Din_n)$ and the out-degree sequence by $\subDout_n = (\Dout_1,\ldots, \Dout_n)$. We call $(\subDout_n,\subDin_n)$ the {\em bi-degree sequence} of the graph.

For a precise description of DCM, we refer to \cite{Olvera,Litvak3}. 
The graph is defined as follows: let $n\in\N$ be the size of the graph, and fix a bi-degree sequence $(\subDout_n,\subDin_n)$. The graph is generated by fixing a free outgoing half edge and we pair it uniformly at random with a free incoming half edge. In this process, self loops and multiple edges are allowed. Until the pairing is made uniformly, it is not relevant in which order we choose the free outgoing half-edge. In this setting, the total in-degree and the out-degree of the graph have to be equal. In the case of random in- and out-degrees, this is a rare event. The algorithm  presented in \cite{Olvera} generates an admissible bi-degree sequence  in a finite number of steps, and approximates the initial degree distributions.

\begin{Condition}[Bi-degree regularity conditions]
\label{cond:bideg:regular}
Let $(\subDout_n,\subDin_n)$ be a bi-degree sequence. Then, the {\em bi-degree regularity conditions} are as follows:
\begin{itemize}
\item[{(a)}] There exists a distribution  $(p(h,l))_{h,l\in\N}$ such that, for every $h,l\in\N$, as $n\rightarrow\infty$, 
	\eqn{
	\label{for:DCM:degcond1}
		\frac{1}{n}\sum_{i\in[n]} \I_{\{\Dout_i=h,\Din_i=l\}}\longrightarrow p(h,l);
	}
\item[{(b)}] Denote by $(\DDout,\DDin)$ a pair of random variables with distribution $(p(h,l))_{h,l\in\N}$ as in \eqref{for:DCM:degcond1}. Then, as $n\rightarrow\infty$, 
\eqn{
	\label{for:DCM:degcond2}
	\frac{1}{n}\sum_{i\in[n]}h\I_{\{\Dout_i=h\}}\longrightarrow \E\left[\DDout\right], \quad\quad \frac{1}{n}\sum_{i=1}^n l\I_{\{\Din_i=l\}}\longrightarrow \E\left[\DDin\right],
}
and $\E\left[\DDin\right]= \E\left[\DDout\right]$;
\item[{(c)}] For $L_n = \Dout_1+\cdots+\Dout_n$, as $n\rightarrow\infty$,
\eqn{
\label{for:DCM:degcon3}
	\frac{1}{n}\sum_{i\in[n]} \frac{h}{L_n}\I_{\{\Dout_i=h,\Din_i = l\}}\longrightarrow \frac{k}{\E[\DDout]}p(h,l) =: p^\star(h,l).
}  
Denote by $(\D^{\star{\sss(\mathrm{out})}},\DDin)$ a pair of random variable with distribution $(p^\star(h,l))_{h,l\in\N}$.
\end{itemize}
\end{Condition}

Condition \ref{cond:bideg:regular}(a) implies that the empirical bi-degree distribution converges to a limiting distribution given by $(p(h,l))_{h,l\in\N}$ as in \eqref{for:DCM:degcond1}. 
Condition \ref{cond:bideg:regular}(b) implies that both the in- and out-degree distributions have finite first moment, equal to the one of $(p(h,l))_{h,l\in\N}$. Condition \ref{cond:bideg:regular}(c) implies that the out-degree size-biased distribution converges to a limiting distribution $(p^\star(h,l))_{h,l\in\N}$ as in \eqref{for:DCM:degcon3}.

With Condition \ref{cond:bideg:regular}, we are ready to state the convergence result on DCM:
\begin{Proposition}
\label{prop:CDM:limit}
Consider a directed configuration model $\mathrm{DCM}_n$ such that the bi-degree sequence $(\subDout_n,\subDin_n)$ satisfies Condition \ref{cond:bideg:regular}. Then, $\mathrm{DCM}_n$ {\em converges in probability in the directed LW sense} to the law $\P$ of a {\em marked} Galton-Watson tree, where
\begin{enumerate}
\item edges are directed from children to parents;
\item  the mark and the in-degree of the root are distributed as $(\DDout,\DDin)$ as in \eqref{for:DCM:degcond1};
\item the mark and the in-degree of any other vertex are independent across the tree vertices, and are distributed according to $(\mathcal{D}^{\star{\sss(\mathrm{out})}},\DDin)$ as in \eqref{for:DCM:degcon3}.
\end{enumerate}
\end{Proposition}
The proof of Proposition \ref{prop:CDM:limit} is an adaptation of the proof for the undirected case as presented in \cite[Section 2.2.2]{vdHfleur}. The proof is divided in two parts. First, we use a coupling argument to prove that $\mathrm{DCM}_n$ converges in distribution to the prescribed limit. The second part consists in the application of the second moment method  on the number of vertices in $\mathrm{DCM}_n$ with a fixed finite neighborhood structure, to prove that the number of such vertices is concentrated around its mean.

We start with the coupling argument:
\begin{Lemma}[LW convergence of DCM in distribution]
\label{lem:DCM:coupling}
Fix a finite marked rooted tree $(H,y,M(H))$. Under the assumptions of Proposition~\ref{prop:CDM:limit} there exists a marked Galton-Watson tree $\mathrm{GW}^{(n)}$ such that  
\eqn{
\label{for:CDM:coupling:1}
	\pr\big(U_{\leq k}(V_n) \cong (H,y,M(H))\big)= \pr\left(\mathrm{GW}^{(n)}_{\leq k}\cong (H,y,M(H))\right)+o(1),
}
where $\mathrm{GW}^{(n)}_{\leq k}$ denote the first $k$ generations of $\mathrm{GW}^{(n)}$. Further, $\mathrm{GW}^{(n)}\rightarrow \P$ locally weakly in distribution, where $\P$ is the limit in Proposition \ref{prop:CDM:limit}. As a consequence, $\mathrm{DCM}_n\rightarrow \P$ locally weakly in distribution.

\end{Lemma}
\begin{proof}
We prove that, for every finite $k\in\N$ and $n$ large enough, the $k$-neighborhood of a uniform chosen vertex in $\mathrm{DCM}_n$ has approximately the same distribution as the first $k$ generations of a marked Galton-Watson tree $\mathrm{GW}^{(n)}$, where marks and offspring distributions in $\mathrm{GW}^{(n)}$ depends on $n$. Define $(p_n(h,l))_{h,l\in\N}$ and $(p^\star_n(h,l))_{ih,l\in\N}$ by
\eqn{
\label{for:DCM:empdist}
	p_n(h,l) = \frac{1}{n}\sum_{i\in[n]} \I_{\{\Dout_i=h,\Din_i=l\}}, \quad p_n^*(h,l) = \frac{1}{n}\sum_{i\in[n]} \frac{h}{L_n}\I_{\{\Dout_i=h,\Din_i = l\}},
}
where $L_n = \Dout_1+\cdots+\Dout_n$. 

The coupling is constructed as follows: the mark and the degree of the root both in $U_{\leq k}(V_n)$ and in $\mathrm{GW}^{(n)}$ are chosen  according to the distribution $p_n$ as in \eqref{for:DCM:empdist}. Therefore, $U_{\leq 0}(V_n)$ and the $0$-generation of $\mathrm{GW}^{(n)}$ (which both consist only of the root and its mark) are the same. 

We have to construct $U_{\leq k}(V_n)$ and  $\mathrm{GW}^{(n)}_{\leq k}$ at the same time.
Conditioning on $U_{\leq k-1}(V_n)$ and $\mathrm{GW}_{\leq k-1}^{(n)}$, the new exploration step from $U_{\leq k-1}(V_n)$ to $U_{\leq k}(V_n)$ is made as follows: assuming that during the exploration up to distance $k-1$ we have created $t$ edges,  take the first unpaired incoming half-edge $x_{t+1}$, that we pair to a uniformly chosen outgoing half-edge that is not paired yet. We choose this outgoing half-edge $y_{t+1}$ uniformly at random among all outgoing half-edges, independently from the previously matched half-edges. 

Let $W_{t+1}$ be the vertex  in $\mathrm{DCM}_n$ to which $y_{t+1}$ is incident. Then, in $\mathrm{GW}^{(n)}_{\leq k}$ we assign to a new vertex mark and in-degree equal to $(\Dout_{W_{t+1}}, \Din_{W_{t+1}})$. Notice that in this case the pair $(\Dout_{W_{t+1}}, \Din_{W_{t+1}})$ is distributed as $p_n^\star$ given in \eqref{for:DCM:empdist}. 

 In $U_{\leq k-1}(V_n)$ we have to be careful since the half-edge $y_{t+1}$ might have already been paired. If $y_{t+1}$ has not been paired yet, then we pair $x_{t+1}$ to $y_{t+1}$ to create an edge. If $y_{t+1}$ has already been paired, then we draw a new outgoing half-edge $y'_{t+1}$ chosen uniformly from the unpaired ones. 
 
 We do this procedure for every ingoing half-edge $x_{t+1},\ldots, x_{t+s}$, where $s$ is the number of unpaired ingoing half-edges in $U_{\leq k-1}(V_n)$.  We can  have differences between the exploration process in $\mathrm{DCM}_n$ and  $\mathrm{GW}^{(n)}$. Differences can happen in two ways:
\begin{enumerate}
	\item the outgoing half-edge that we select to create a new edge has already been paired;
	\item the outgoing half-edge that we select to create a new edge has not been paired yet, but it is incident to a vertex already found in the exploration process.
\end{enumerate}
These two contributions have small probability. In fact, after creating $t$ edges, the probability that we select an outgoing half-edge that is already used is equal $ t/L_n$, where $L_n$ is the total number of outgoing edges. This means that the probability that in the first $s$ steps  we use the same out-going half-edge twice is bounded by
\eqn{
\label{for:CDM:reuseedges}
	\sum_{t=0}^s \frac{t}{L_n} = \frac{s(s+1)}{2L_n}.
}
Thanks to Condition \ref{cond:bideg:regular}(b), $L_n$ is of order $n$, so the expression in \eqref{for:CDM:reuseedges} is $o(1)$ whenever $s = o(\sqrt{n})$. The probability of selecting a vertex $i$ when choosing an outgoing half-edge is $\Dout_i/L_n$. Then, the probability that a vertex $i$ is selected at least twice when $t$ edges are created is bounded by
\eqn{
\label{for:CDM:reuseV:1}
	\frac{t(t+1)}{2} \frac{(\Dout_i)^2}{L_n^2}
}
Using \eqref{for:CDM:reuseV:1} and the union bound, the probability that a vertex is selected twice when $T$ edges are created is bounded by
\eqn{
\label{for:CDM:reuseV:2}
	\frac{t(t+1)}{2}\sum_{i=1}^n\frac{(\Dout_i)^2}{L_n^2} \leq \frac{t(t+1)}{2L_n}\Dout_{max},
}
where $\Dout_{max}$ is the maximum out-degree in the bi-degree sequence. In this case, the expression in \eqref{for:CDM:reuseV:2} is $o(1)$ when $s= o(\sqrt{n/\Dout_{max}})$. Further, $\Dout_{max}$ under Condition \ref{cond:bideg:regular} is $o(n)$.

The two bounds in \eqref{for:CDM:reuseedges} and \eqref{for:CDM:reuseV:2} together holds for $s= s(n)$, with $s(n)\rightarrow \infty$  sufficiently slowly.
Since any finite tree $H$ is made by a finite number of edges $S$, we can take $n$ large enough such that $s(n)\geq S$. This implies (\ref{for:CDM:coupling:1}).
 Note that from (\ref{for:CDM:coupling:1}) it directly follows that
	$$
		\pr\left(U_{\leq k}(V_n) \cong \mathrm{GW}^{(n)}_{\leq k}\right) = 1-o(1).
	$$
Finally, since the distributions $p_n$ and $p_n^\star$ converge respectively to $p$ and $p^\star$ as defined in Condition \ref{cond:bideg:regular}, and \eqref{for:CDM:coupling:1} holds for any finite marked rooted tree $(H,y,M(H))$, we have proved that $\mathrm{DCM}_n$ converges locally weakly in distribution to ${\cal P}$.
\end{proof}

Next we  prove the convergence in probability, using the second moment method on the number of vertices in $\mathrm{DCM}_n$ with a prescribed neighborhood $(H,y,M(H))$. 

\begin{Lemma}[Second moment method]
Fix $k\in\N$ and a finite structure $(H,y,M(H))$ for the root neighborhood. Let $N_k(H,y,M(H))$ be the number of vertices $i$ in $\mathrm{DCM}_n$ such that $U_{\leq k}(i) \cong (H,y,M(H))$. Then, as $n\rightarrow\infty$, 
\eqn{
\label{for:DCM:mom1}
	\frac{1}{n^2}\E\left[N_k(H,y,M(H))^2\right] \longrightarrow \P\left(U_{\leq k}(\emp)\cong (H,y,M(H))\right)^2.
}
\end{Lemma}
\begin{proof}
We can rewrite 
$$
\E\left[N_k(H,y,M(H))^2\right]/n^2 = \pr\left(U_{\leq k}(V_n^1)\cong(H,y,M(H)), U_{\leq k}(V_n^2)\cong (H,y,M(H))\right),
$$
where $V_n^1$ and $V_n^2$ are two vertices chosen uniformly at random in $\mathrm{DCM}_n$. Since we fix $k\in\N$, we can take $n$ large enough such that, with high probability, $V^2_n$ is not a vertex found in the exploration up to distance $2k$ from $V^1_n$. Then we can rewrite the probability in the right-hand side of \eqref{for:DCM:mom1} as 
$$
	\pr\left(U_{\leq k}(V_n^1)\cong(H,y,M(H)), U_{\leq k}(V_n^2)\cong(H,y,M(H)), V^2_n\not\in U_{\leq 2k}(V_n^1) \right)+ o(1),
$$
where the factor $2k$ comes from the fact that we look at the structure $(H,y,M(H))$ for the two neighborhoods when they are disjoint.
With a similar argument to the one just used, since $k$ is fixed,
\eqn{
\label{for:DCM:mom2}
	\pr\left(U_{\leq k}(V_n^1)\cong(H,y,M(H)),  V^2_n\not\in U_{\leq 2k}(V_n^1) \right)\longrightarrow \P\left(U_{\leq k}(\emp)\cong (H,y,M(H))\right).
}
We now use the fact that, conditioning on the existence of a tree in $\mathrm{DCM}_n$, the probability to have a second tree disjoint from the first one is equal to have a tree in a different configuration model with different size and bi-degree distribution. 
More precisely, conditioning on $\{U_{\leq k}(V_n^1)\cong(H,y,M(H)),  V^2_n\not\in U_{\leq 2k}(V_n^1)\}$,  we want to evaluate the probability of having a second tree $U_{\leq k}(V_n^2)\cong(H,y,M(H))$, disjoint from $U_{\leq k}(V_n^1)\cong(H,y,M(H))$. We have that 
\eqn{
\label{for:DCM:mom3}
\begin{split}
&\pr\left(\left. U_{\leq k}(V_n^2)\cong(H,y,M(H))\right| U_{\leq k}(V_n^1)\cong(H,y,M(H)), V^2_n\not\in U_{\leq 2k}(V_n^1) \right)\\
&=\pr\left( \widehat{U}_{\leq k}(\widehat{V}_n^2)\cong(H,y,M(H)),\ \widehat{j}\not\in\widehat{U}_{\leq k}(\widehat{V}_n^2)\right),
\end{split}
}
where $\widehat{U}_{\leq k}(\widehat{V}_n^2)$ is the $k$-neighborhood of a vertex $\widehat{V}_n^2$ chosen uniformly at random in a different configuration model $\widehat{\mathrm{DCM}}_n$, and $\widehat{j}$ is a particular vertex in $\widehat{\mathrm{DCM}}_n$ whose characteristics are specified below. 

The vertices set and bi-degree sequence of $\widehat{\mathrm{DCM}}_n$ are defined as follows:
\begin{enumerate}
	\item if $i\not\in U_{\leq k}(V_n^1)$, then $i$ is a vertex in  $\widehat{\mathrm{DCM}}_n$ with the same in- and out-degree $(\Dout_i,\Din_i)$;
	\item if $i\in U_{\leq k}(V_n^1)$, then $i$ is not present in $\widehat{\mathrm{DCM}}_n$;
	\item define an additional vertex $\widehat{j}$ in $\widehat{\mathrm{DCM}}_n$, with in- and out-degree $(\widehat{\Dout}_j,\widehat{\Din}_j)$, where $\widehat{\Dout}_j$ equals the sum of the unpaired outgoing half-edges in $ U_{\leq k}(V_n^1)$, and $\widehat{\Din}_j$ equals the number of unpaired ingoing half-edges in $ U_{\leq k}(V_n^1)$. We point out that $\widehat{U}_{\leq k}(\widehat{V}_n^2)$ needs to avoid $\widehat{j}$. 
\end{enumerate}
Notice that the unpaired incoming half-edges in $ U_{\leq k}(V_n^1)$ are incident only to vertices at distance $k$ from the root, while the unpaired outgoing half-edges are incident to all vertices in $ U_{\leq k}(V_n^1)$. We have that $\widehat{\mathrm{DCM}}_n$ is a graph with $n-|U_{\leq k}(V_n^1)|+1$ vertices, and a different bi-degree sequence. 

The graph $\widehat{\mathrm{DCM}}_n$ is then created by pairing an incoming half-edge to a uniformly chosen outgoing half-edge, as usual as in the regular $\mathrm{DCM}_n$. The probability to observe a structure in $\widehat{\mathrm{DCM}}_n$ that is disjoint from the vertex $\widehat{j}$ is exactly the same as in the regular $\mathrm{DCM}_n$, conditioning on the structure of $ U_{\leq k}(V_n^1)$. This explain the equality in \eqref{for:DCM:mom3}. 

It is immediate to verify that the  bi-degree sequence of $\widehat{\mathrm{DCM}}_n$ satisfies Condition \ref{cond:bideg:regular}, since we modify a negligible fraction of vertices (recall that $k$ is fixed). As a consequence, 
\eqn{
\label{for:DCM:mom4}
	\pr\left( \widehat{U}_{\leq k}(\widehat{V}_n^2)\cong(H,y,M(H))\right)\longrightarrow  \P\left(U_{\leq k}(\emp)\cong (H,y,M(H))\right).
}
Using together \eqref{for:DCM:mom2} and \eqref{for:DCM:mom4}, we complete the proof of \eqref{for:DCM:mom1}.
\end{proof}

\paragraph{DCM with independent in- and out-degrees.} In \cite{Litvak3} the limiting distribution of PageRank in DCM has been obtained when the size-biased in- and out-degrees are independent:
\[
p^\star(h,l) = \frac{h}{\E[\DDout]}\pr\left(\D^{\star{\sss(out)}}=h\right)\pr\left(\DDin=l\right).
\]
Notice that $\DDout$ and $\DDin$ can, in general, be dependent, that is, $\DDin$ may have a different distribution conditioned on the event $\{\DDout\neq0\}$, because the vertices with zero out-degrees do not contribute in PageRank of other vertices. 

The local weak convergence for this case follows from \cite[Lemma 5.4]{Litvak3}, hence, our Theorem~\ref{th:pagerank} provides an alternative argument for the existence of the limiting PageRank distribution. 
It has been proved in \cite{Litvak3}, under some technical assumptions, that in the limit the PageRank is distributed as
\eqn{
\label{for:DCM:indep1}
	\mathcal{R} \stackrel{d}{=} \sum_{i=1}^{\mathcal{N}}\frac{c}{\D^{\star{\sss(\mathrm{out})}}_i}\mathcal{R}^\star_i+(1-c),
}
where $\mathcal{R^\star}$ are independent realizations of the endogenous solution of the stochastic fixed-point equation 
\eqn{
\label{for:DCM:indep2}
	\mathcal{R^\star} \stackrel{d}{=} \sum_{i=1}^{\mathcal{N^\star}}\frac{c}{\D^{\star{\sss(\mathrm{out})}}_i}\mathcal{R}^\star_i+(1-c).	
}
The recursion (\ref{for:DCM:indep2}) has been studied in a number of papers, see \cite{Jelenkovic2010WBP,Litvak4}, and further references in \cite{Litvak3}. The argument in \cite{Litvak3} is more general, in fact the authors consider generalized PageRank as solution of a more general equation than \eqref{for:DCM:indep2}, where the $(1-c)$ is replaced by a random variable ${\cal B}$. In particular, if $\DDin$ is regularly varying with a tail heavier than the tail of ${\cal B}$, then the limiting PageRank ${\cal R}$  follows a power law with the same exponent as the in-degree $\DDin$.

\subsection{Inhomogeneous random graphs}
\label{sec-ex-IRG} 
In the directed inhomogeneous random graphs, each vertex $i$ receives an in-weight $\Win_i$ and an out-weight $\Wout_i$. There is a directed edge from vertex $i$ to vertex $j$ with probability $w_{ij}^{(n)}$, which depends on $\Wout_i$ and $\Win_j$. Lee and Olvera-Cravioto~\cite{Lee2017PR-IRG} study PageRank in the class of inhomogeneous random graphs that satisfy the assumption
\[w_{ij}^{(n)}=\min\left\{1,\frac{\Wout_i\Win_j}{\theta\,n}(1+\phi_{ij}(n))\right\},\]
 where $\phi_{ij}(n)$ satisfies some technical conditions, and is in fact vanishing as $n\to\infty$ for most natural models. This formulation includes Erd\H{o}s-R\'enyi model, the Chung-Lu model, the Poissonian
random graph and the generalized random graph. For a detailed analysis of the properties of such directed graphs we refer to~\cite{Cao2017IGR}.

LWC for this class of graphs follows directly from \cite[Theorem 3.6]{Lee2017PR-IRG} under general conditions, including that the in- and out-weights are allowed to be dependent. Hence, our results imply that PageRank converges in this model as well, to the PageRank of the limiting random graph. 

In the case when the in- and out-weights are asymptotically independent, it is proved in \cite{Lee2017PR-IRG} that the PageRank converges to the attracting endogenous solution of stochastic recursion (\ref{for:DCM:indep2}). In particular, a power-law distribution of in-weights implies the power-law distribution of PageRank.

\subsection{Directed CTBP trees}
\label{sec:ex:CTBP}
CTBPs are models that describe the evolution of a population composed by individuals that produce children according to i.i.d.\ birth processes. These models have been intensively studied in the literature \cite{athrBook,Jagers,Nerman}.
The convergence result is stated in Proposition \ref{prop:CTBP:limit} below, which requires some notation from CTBPs theory that we present now.

\begin{Definition}[Branching process]
\label{def-BP}
We define the {\em Ulam-Harris set} as
\eqn{
\label{for:Ulamset}
	\U = \bigcup_{n\in\N}\N^n, \quad \quad \mbox{where}\quad \N^0 := \{\emp\}.
}
Consider a birth process $\xi$. Then, the continuous-time branching process is described by
\eqn{
\label{space}
	(\Omega,\mathcal{A},\pr) = \prod_{x\in\U}(\Omega_x,\mathcal{A}_x,\pr_x),
}
where $(\Omega_x,\mathcal{A}_x,\pr_x)$ are probability spaces and $(\xi^x)_{x\in\U}$ are i.i.d.\ copies of $\xi$. For $x\in\N^n$ and $k\in\N$ we denote the $k$th child of $x$ by $xk\in\N^{n+1}$. More generally, for $x\in\N^n$ and $y\in\N^m$, we denote the $y$ descendant of $x$ by $xy$. We call the branching process the triplet $(\Omega,\mathcal{A},\pr)$ and the sequence of point processes $(\xi^x)_{x\in\U}$. We denote the branching process by $\sub{\xi}$.
\end{Definition}

The behavior of CTBPs is determined by properties of the birth process.
Consider a jump process $\xi$ on $\R^+$, i.e., an integer-valued random measure on $\R^+$.  Then we say that $\xi$ is {\em supercritical and Malthusian} when there exists $\alpha^*>0$ such that
\eqn{
\label{for:supercr}
	\E\left[\xi_{T_{\alpha^*}}\right]=1, \quad \quad \mu:= -\left.\frac{d}{d\alpha}\E\left[\xi_{T_{\alpha}}\right]\right|_{\alpha = \alpha^*}<+\infty,
}
where $T_{\alpha}$ is an exponential random variable with mean $1/\alpha$. The unique value $\alpha^*$ that satisfies \eqref{for:supercr} is called the {\em Malthusian parameter}.

An important class of functions of branching processes are {\em random characteristics}. 
A {\em random characteristic} is a real-valued process $\Phi\colon \Omega\times\R\rightarrow\R$ such that, for $x\in\U$, $\Phi(x,s)=0$ for any $s<0$, and $\Phi(x,s) = \Phi(s)$ is a deterministic bounded function for every $s\geq 0$ that only depends on $x$ through the birth time of the individual, its birth process as well as the birth processes of its children. 

Random characteristics are used to evaluate the number of individuals that at time $t\geq0$ satisfies a property. For instance, consider  $\Phi(t)= \I_{\R^+}(t)$ for $x\in\U$ and $t\geq0$, i.e., the characteristic that is equal to one whenever the individual is alive at time $t$. Then the branching process evaluated at time $t$ with the random characteristic $\I_{\R^+}(\cdot)$ is equal to the number of individual alive at time $t$. We denote the CTBP evaluated with a random characteristic $\Phi$ by $\sub{\xi}^\Phi_t$.

It is known that, for a random characteristic $\Phi$, as $t\rightarrow\infty$,
\eqn{
\label{th-expogrowth-f2}
	\frac{\sub{\xi}^{\Phi}_t}{\sub{\xi}^{\I_{\R^+}}_t}
		\stackrel{\pr-\mbox{a.s.}}{\longrightarrow}
			\E\left[\Phi(T_{\alpha^*})\right],
}
where the left-hand term in \eqref{th-expogrowth-f2} is the fraction of alive individuals that satisfies the property given by $\Phi$. The right-hand side is the expectation of $\Phi$, evaluated at an exponentially distributed time $T_{\alpha^*}$, on an independent copy of the CTBP. 

The convergence in \eqref{th-expogrowth-f2} is a general result that is used often in the literature \cite{Athr2,GarvdHW,Jagers,Nerman,RudValko}. We refer to \cite[Theorem A]{RudValko} for a simplified formulation of the result contained in \cite{Nerman}.

With the notation just introduced we can formulate the convergence result:
\begin{Proposition}[LWC for CTBPs trees]
\label{prop:CTBP:limit}
Consider a supercritical and Malthusian birth process $(\xi_t)_{t\geq0}$. Denote the corresponding CTBP by $\sub{\xi}$. Let $\T(t)$ be the directed random tree defined by $\sub{\xi}$ at time $t$, where edges are directed from children to parents. Then, on the event $\{|\T(t)|\rightarrow\infty\}$, $\T(t)$ {\em converges $\pr$-a.s. in the LW sense} to the law of $\T(T_{\alpha^*})$, where
\begin{enumerate}
	\item all marks are 1;
	\item edges are directed from children to parents;
	\item $T_{\alpha^*}$ is an exponentially distributed random variable with parameter $\alpha^*$ (the Mathusian parameter of the CTBP).
\end{enumerate} 
\end{Proposition}
\begin{proof}
First of all, at every $t\in\R^+$, $\T(t)$ is a directed finite tree. We can equivalently prove the result on the discrete sequence $(\T_n)_{n\in\N}$, where $\T_n = \T(\tau_n)$, for $(\tau_n)_{n\in\N}$ the sequence of birth times of the CTBP.

Denote the vertices in $\T_n$ by their birth order, which means that the root of $\T_n$ in the sense of CTBP is vertex $1$. First of all, notice that, for every $i\in[n]$ and $N\in\N$, the $N$ neighborhood $U_{\leq N}(i)$ in the directed marked rooted graphs $(\T_n,i,1)$  is just the subtree rooted at $i$ composed by the descendents of $i$ only up to generation $N$ (from $i$). Notice that every vertex has out-degree $1$ except for vertex $1$ since it has out-degree 0. 

What we need to prove is that, for any finite directed rooted tree $(H,y)$ of depth $N$ and with mark 1 for every vertex, we have, as $n\rightarrow\infty$, 
\eqn{
	\frac{1}{n}\sum_{i\in[n]}\I\left\{ U_{\leq N}(i) \cong (H,y) \right\} \stackrel{\pr-a.s.}{\longrightarrow} \pr\left(U_{\leq N}(\emp) \cong (H,y)\right),
}
where $U_{\leq N}(\emp)$ is the $N$-neighborhood of the root $\emp$ in the random tree $\T(T_{\alpha^*})$. For every $i\in[n]$ the indicator function inside the expectation satisfies the definition of random characteristic, since it is a bounded function that, for every individual $i$ in the branching population, depends only on the birth time $\tau_i$ and on the randomness associated to $i$ and its descendants. As a consequence, the result follows by \eqref{th-expogrowth-f2}.
\end{proof}

This result resembles the subtree counting result in \cite[Theorem 2]{RudValko}.
Notice that the limiting rooted graph in Proposition \ref{prop:CTBP:limit} is finite with probability 1. This is rather different than the undirected settings, where typically the limiting rooted graph is infinite when considering a sequence of graphs with growing size.

\begin{Remark}[{\bfseries Non-recursive property of PageRank}]
\label{rem:nonrecursive:CTBP}
{\em 
the behavior of PageRank is often investigated starting from the recursive distributional equation in \eqref{for:recursive:pagerank:pers:norm}. In particular, the solution of  \eqref{for:recursive:pagerank:pers:norm} is constructed using a weighted Galton-Watson tree. This construction is based on the fact that the subtree rooted at every vertex is again a Galton-Watson tree with the same distribution.

In some cases, the construction is adapted to allow the root to have different degree and mark, but all other vertices have i.i.d.\ characteristics. As an example, we refer to \cite{Litvak3}, where PageRank on directed configuration model is investigated (in the independent case, see Section \ref{sec:ex:DCM}). 

When we consider CTBPs, we have proved that the graph-normalized PageRank converges to the PageRank value of the root in a tree with distribution $\T(T_{\alpha^*})$. In particular, the processes $\{(\xi_t)_{t\geq0}^x\}_{x\in\U}$ that define $\T(T_{\alpha^*})$ are i.i.d., but they are evaluated at random dependent times $(T_{\alpha^*}-\tau_x)_{x\in\U}$. Thus, the solution based on a weighted Galton-Watson tree does not apply to the PageRank in CTBPs, as the CTBP is inhomogeneous.
}
\end{Remark}

\subsection{Preferential attachment model}
\label{sec:examples:PAM}
Preferential attachment models (PAMs) are discrete-time dynamical models of random graphs. The main idea behind these models is the following: conditioning on the actual state of the graph, a new vertex is added with one (or more) edges, that are attached to existing vertices with probabilities proportional to their degree plus a constant.

There are different  possible definitions of the model. See \cite{BergerBorgs,DSMCM} as well as \cite[Chapter 8]{vdH1} for different definitions of the model. We consider a modification of the {\em sequential} PAM as presented in \cite{BergerBorgs}. Fix $m\geq 1$ to be the initial degree of the vertices, and a constant $\delta>-m$. Then, we define a sequence of graphs $(\PA_n(m,\delta))_{n\in\N}$ as follows:
\begin{enumerate}
	\item for $n=1$, $\PA_1(m,\delta)$ is composed by a single vertex with no edges;
	\item for $n=2$, $\PA_2(m,\delta)$ is composed by two vertices with $m$ edges between them;
	\item for $n\geq 3$,  $\PA_n(m,\delta)$ is defined recursively: we add a vertex to $\PA_{n-1}(m,\delta)$ with $m$ edges. These $m$ edges are attached to existing vertices with the following probability: for $l=1,\ldots, m$,
		\eqn{
	\label{for:PA:rule}
		\pr\left(\left.n\stackrel{l}{\rightarrow}i \right| \PA_{n-1,l-1}(m,\delta)\right) = 
\frac{D_i(n-1,l-1)+\delta}{2m(n-2)+ (n-1)\delta+(l-1)}.
	}
	In \eqref{for:PA:rule}, $D_i(n-1,l-1)$ denotes the degree of vertex $i$ in the graph of size $n-1$ and after the first $l-1$ edges of the new vertex have been attached. 
\end{enumerate}
Notice that we allow for multiple edges but not for  self-loops.
In this case we talk about PAM with {\em affine attachment rule}, since the attachment probabilities are proportional to an affine function of the degree. This model was first introduced in \cite{ABrB} for $\delta=0$.  PAMs have gained a lot of attention in the last years since they show properties found in many real-world networks. In fact, PAMs shows a power-law degree distribution with exponent $\tau = 3+\delta/m$ \cite[Section 8.4]{vdH1}, and they shows the {\em small-world phenomenon}, i.e., the typical distance and the diameter of the graph are small compared to the size of the graph itself 
\cite{CarGarHof,DSMCM,DSvdH}.

It is known that CTBPs can embedd PAMs in continuous-time \cite{Athr,Athr2,RudValko}. We give a definition of the birth process that describes PAMs:
\begin{Definition}[Embedding birth process]
\label{def-emb_birthpr}
Fix $m\geq 2$ and $\delta>-m$. Consider the sequence $(k+1+\delta/m)_{k\in\N}$. Let $(E_k)_{k\in\N}$ be a sequence of independent and exponentially distributed random variables, with $E_k \sim E(k+1+\delta/m)$, and $E_{-1}=0$. We call $(\xi_t)_{t\geq0}$ the {\em embedding birth process}, where $\xi_t = k$ if $t\in[E_{-1}+\cdots+E_{k-1},E_{-1}+\cdots+E_k)$.
\end{Definition}

This construction is already used in \cite{Athr,Athr2,GarvdH2,RudValko}. The embedding holds for any $m\geq 2$, but the topological description of the graph as a CTBP is used only in \cite{GarvdH2}. 

Originally defined as undirected graphs, PAMs have a natural direction from edges given by the recursive definition of such models. We can see every edge as directed from young to old, therefore every vertex in $\PA_t(m,\delta)$ has out-degree $m$. If we see the CTBP defined by the process in \eqref{def-emb_birthpr} as the continuous-time version of the PAM with out-degree 1, then the directed local weak limit is given by Proposition \ref{prop:CTBP:limit}. 

For $m\geq 2$, PAM is no longer a tree, making the analysis harder than the tree case.
In \cite{BergerBorgs}, Berger, Borgs, Chayes and Saberi give the local weak limit in probability for the undirected version of PAM with affine attachment rule when $\delta \geq 0$. 
 When $\delta\in(-m,0)$, we believe that the result holds by adapting the proof in \cite{BergerBorgs}. This argument is left for future work. We give a definition of the limiting graph for DPAMs:

\begin{Definition}[Directed P\'olya point graph]
\label{def:pointgraph:dir}
The {\em directed P\'olya point graph} is an infinite marked rooted random tree constructed as follows: let $m\geq 2$ and $\delta>-m$ be parameters for a preferential attchment model $(\PA_t(m,\delta))_{t\in\N}$. Let
\begin{enumerate}
	\item[{(a)}]  $\chi = (m+\delta)/(2m+\delta)$, $\psi=(1-\chi)/\chi$;
	\item[{(b)}]$\Gamma_{in}$ denote a Gamma distribution with parameters $m+\delta$ and 1;
\end{enumerate}
Vertices in the graph have three characteristics:
\begin{enumerate}
	\item[{(a)}] a {\em label} $i$ in the Ulam-Harris set;
	\item[{(b)}] a {\em position} $x\in[0,1]$;
	\item[{(c)}] a positive number $\gamma$ called {\em strength};
\end{enumerate}
In addition, every vertex has mark $m$ (in the sense of Definition \ref{def:rootneigh:dir}). Assign to $\emp$ a position $x_\emp= U^\chi$, where $U$ is a uniform random variable on $[0,1]$, and a strength $\gamma_\emp\sim \Gamma_{in}$. Set $\emp$ as unexplored. 
Then, recursively over the elements in the set of unexplored vertices, according to the shortlex order:
\begin{enumerate}
	\item let $i$ denote the current unexplored vertex;
	\item assign to $i$ a strength value $\gamma_i \sim \Gamma_{in}$;
	\item let $u_{i1},\ldots,u_{i \Din_i}$ be the random $\Din_i$ points given by an independent Poisson process on $[u_i,1]$ with density
	$$
		 \rho_{i}(x) = \gamma_i\frac{\psi x^{\psi-1}}{x_i^\psi}.
	$$
	\item draw an edge from each one of the vertices $i1,\ldots,i\Din_i$ to $i$;
	\item set $x_{i1},\ldots,x_{i \Din_i}$ unexplored and $i$ as explored.
\end{enumerate}
\end{Definition}
Definition \ref{def:pointgraph:dir} is obtained by the definition of the undirected LW limit of PAM given in \cite[Section 2.3.2]{BergerBorgs}, where the exploration of the neighborhood of a vertex is limited to the exploration of {\em younger vertices}. In other words, the exploration from a vertex $i$ is made only over vertices with index $j>i$. 
The positions in Definition \ref{def:pointgraph:dir} encode the {\em age} of a vertex in PAM. In fact, it is possible to identify a vertex $i\in[t]$ in PAM with the point $(i/t)^\chi$ \cite[Lemma 3.1]{BergerBorgs}, so old vertices have position closer to 0 than young vertices. 

With the definition of the Directed P\'olya point graph, we can state the directed LWC result for PAMs:

\begin{Proposition}[LW limit of directed PAM]
\label{prop:PAM:dirLW}
Fix $m\geq 2$ and $\delta>-m$. Let $(\mathrm{PA}_t(m,\delta))_{t\in\N}$ be a PAM defined by the attachment rule in \eqref{for:PA:rule}. Denote by $(\mathrm{DPA}_t(m,\delta))_{t\in\N}$ the directed version of $(\mathrm{PA}_t(m,\delta))_{t\in\N}$, where edges are directed from young  to old vertices. Then,
\begin{enumerate}
	\item for $\delta\geq0$,  $\mathrm{DPA}_t(m,\delta)$ {\em converges in probability in the directed LW sense} to the directed P\'olya point graph as in Definition \ref{def:pointgraph:dir};
	\item for $\delta\in(-m,0)$, if \cite[Theorem 2.2]{BergerBorgs} can be extended, then the convergence holds also in this case.
	\end{enumerate}
\end{Proposition}
The proof of Proposition \ref{prop:PAM:dirLW} follows immediately from \cite[Theorem 2.2]{BergerBorgs} and the fact that the exploration process in  $\mathrm{DPA}_t$ corresponds to exploring only younger vertices.

\begin{Remark}[{\bfseries Non-recursive property of PageRank}]
{\em 
Similarly to Remark \ref{rem:nonrecursive:CTBP} about CTBPs, we point out that the PageRank value of the root of a directed P\'olya point graph does not satisfy the recursive property that is necessary to consider it as a  solution of \eqref{for:recursive:pagerank:pers:norm}. Notice that the Poisson point process assigned to vertex $i$ in Definition \ref{def:pointgraph:dir} is defined on the interval $[x_i,1]$, where the position $x_i$ depends on the ancestors (in the Ulam-Harris sense) of $i$. 

Another way to interpret this is that the family of Poisson point process in Definition \ref{def:pointgraph:dir} is composed by i.i.d.\ processes parametrized by the positions of vertices, that are dependent random variables. This suggests that the positions in the P\'olya point graph play the same role as the birth times in CTBPs.
}
\end{Remark}

\bigskip

\noindent
{\bfseries Acknowledgments.}
This work is supported in part by the Netherlands Organisation for Scientific Research (NWO) through the Gravitation {\sc Networks} grant 024.002.003. The work of RvdH is further supported by the Netherlands Organisation for Scientific Research (NWO) through VICI grant 639.033.806.


\printbibliography[title=References, heading = bibintoc]

\end{document}